\documentclass[12pt]{amsart}
\usepackage{geometry}
\usepackage{amsthm, amssymb, thm-restate}
\usepackage{esint}
\usepackage{color}
\usepackage{array}
\usepackage{booktabs}
\usepackage{multirow}
\usepackage{graphicx}

\makeatletter

\numberwithin{equation}{section}
\numberwithin{figure}{section}

\makeatother

\newcommand{\head}[1]{\textnormal{\textbf{#1}}}

\newcommand{\be}{\begin{equation}}
\newcommand{\ee}{\end{equation}}
\newcommand{\NN}{\mathbb{N}}
\newcommand{\RR}{\mathbb{R}}
\newcommand{\Hess}{\texttt{Hess}}
\newcommand{\Ric}{\texttt{Ric}}

\newtheorem{thm}{Theorem}

\newtheorem{prop}[thm]{Proposition}
\newtheorem{lem}[thm]{Lemma}
\newtheorem*{lem*}{Lemma}
\newtheorem*{thm*}{Theorem}
\newtheorem{cor}[thm]{Corollary}

\newtheorem{Assumption}[thm]{Assumption}

\begin{document}

\title[Finite VDM]{Embeddings of Riemannian Manifolds with Finite Eigenvector Fields of Connection Laplacian} 

\author{Chen-Yun Lin}
\address{Chen-Yun Lin\\
Department of Mathematics\\
University of Toronto}
\email{cylin@math.toronto.edu}

\author{Hau-Tieng Wu}
\address{Hau-Tieng Wu\\
Department of Mathematics\\
University of Toronto}
\email{hauwu@math.toronto.edu}


\maketitle

\begin{abstract}
We study the problem asking if one can embed manifolds into finite dimensional Euclidean spaces by taking finite number of eigenvector fields of the connection Laplacian. This problem is essential for the dimension reduction problem in massive data analysis. Singer-Wu proposed the vector diffusion map which embeds manifolds into the Hilbert space $l^2$ using eigenvectors of connection Laplacian. In this paper, we provide a positive answer to the problem. Specifically, we use eigenvector fields to construct local coordinate charts with low distortion, and show that the distortion constants depend only on geometric properties of manifolds with metrics in the little H\"{o}lder space $c^{2,\alpha}$. Next, we use the coordinate charts to embed the entire manifold into a finite dimensional Euclidean space. The proof of the results relies on solving the elliptic system and provide estimates for eigenvector fields and the heat kernel and their gradients. We also provide approximation results for eigenvector field under the $c^{2,\alpha}$ perturbation.
\end{abstract}

\newpage

\section{Introduction}
\label{intro}
In the past decade, there is a growing consensus that high dimensional and massive dataset analysis is a key to future advances. Although frequently the information regarding the structure underlying a dataset is limited, we generally believe that a ``lower dimensional'' and possibly nonlinear structure should exist. 

A widely accepted approach to model the low dimension structure is by considering the manifold; that is, we assume that the collected dataset, while might be of high dimension, is located on a low dimensional manifold \cite{Tenenbaum_deSilva_Langford:2000,Belkin_Niyogi:2003,Coifman_Lafon:2006}. How to analyze the dataset under this assumption is generally called the {\it manifold learning} problem, and one particular goal is to recover the nonlinear low dimensional structure of the manifold. Under this assumption, several algorithms were proposed toward this goal, like isomap \cite{Tenenbaum_deSilva_Langford:2000}, locally linear embedding (LLE) \cite{Roweis_Saul:2000}, eigenmaps (EM) \cite{Belkin_Niyogi:2003}, diffusion maps (DM) \cite{Coifman_Lafon:2006}, Hessian LLE \cite{Donoho_Grimes:2003}, vector diffusion maps \cite{Singer_Wu:2012,Singer_Wu:2014}, nonlinear independent component analysis or empirical intrinsic geometry \cite{Coifman_Singer:2008,Talmon_Coifman:2013}, alternating diffusion \cite{Lederman_Talmon:2015,Talmon_Wu:2015}, etc. In a nutshell, by taking the local information of the given point cloud, like geodesic distance, and by applying knowledge from the spectral geometry and index theory, we could obtain information of the manifold from different aspects, like the parametrization of the dataset or the topological feature of the dataset. 

One particular interesting problem is how to guarantee the data visualization and/or dimensional reduction. {\it Data visualization} problem is asking if we could embed a given manifold into the three dimensional Euclidean space, so that we could visualize the dataset; {\it dimension reduction} problem is asking if we could embed a given manifold into a low dimensional Euclidean space, which dimension is smaller than that of the dataset. Mathematically, this problem is formulated as asking if it is possible to embed the manifold (hence the dataset) into a finite dimensional Euclidean space, even isometrically.  The embedding problem was first positively answered by Whitney \cite{Whitney:1944}, and the isometrically embedding problem was first solved by Nash \cite{Nash:1956}. However, the approach by Nash, the implicit function theory, is not canonical and is not essentially feasible for data analysis. In Berard, Besson and Gallot's breakthrough paper \cite{Berard_Besson_Gallot:1994}, the spectral embedding idea was explored to answer this kind of problem. The main idea is embedding the Riemannian manifold by the eigenfunctions of the associated Laplace-Beltrami operator via studying the associated heat kernel behavior. We could show that the spectral embedding is a canonical embedding in the sense that it depends only on the eigenfunctions and eigenvalues, and the result is close to isometric with error of the order of the diffusion time. In \cite{Wang_Zhu:2015}, this embedding was further modified to an almost isometric embedding, with error up to any give finite order of the diffusion time, by applying the implicit function theory. 

The spectral embedding idea is directly related to many manifold learning algorithms, like EM and DM. While these algorithms work, however, numerically we are able to obtain only finite eigenfunctions and eigenvalues in practice, while theoretically all the eigenfunctions and eigenvalues of the Laplace-Beltrami operator, which are countably infinite, are needed to study the spectral embedding. Thus, the next natural question we could ask is the possibility to embed the manifold by the finite eigenfunctions and eigenvalues. This question was positively answered separately by Bates \cite{Bates:2014} and Portegies \cite{Portegies:2015}; that is, one is able to embed the manifold with finite eigenfunctions and eigenvalues. In Portegies, it is further shown that the embedding could be almost isometric with a prescribed error bound. 

The work in Bates is essentially based on the local parametrization work reported in Johns, Maggioni and Schul \cite{Jones_Maggioni_Schul:2010}, where the inherited oscillatory behavior of the eigenfunction is taken into account to guarantee that locally we could find finite eigenfunctions so that we could embed a local ball via these finite eigenfunctions with low distortion. What needs to be proved is that the embedding of different local balls will not intersect each other. These works fundamentally answer why EM and DM could work well in practice.

The spectral embedding mentioned above depends on the Laplace-Beltrami operator. The vector diffusion map (VDM) \cite{Singer_Wu:2012,Singer_Wu:2014}, on the other hand, depends on the connection Laplacian associated with a possibly nontrivial bundle structure. In brief, the VDM with the diffusion time $t>0$ is defined by the eigenvector fields of the connection Laplacian by 
\begin{align}
V_t: M&\rightarrow \ell^2\label{Definition:VDM}\\
x &\mapsto \left(e^{-(\lambda_i+\lambda_j)t/2}\langle X_{i}, X_{j}\rangle \right)_{i,j=1}^\infty\,,\nonumber
\end{align}
where $x\in M$ and $X_i$ is the $i$-th eigenvector field of the connection Laplacian associated with the eigenvalue $\lambda_i$.
The basic properties of $V_t$ have been shown in \cite{Singer_Wu:2012,Singer_Wu:2014,Wu:2014}. For example, we could see that the VDM is an embedding, and it is close to an isometric embedding with error depending on the diffusion time.
The VDM is originally motivated by studying the cryo electron microscope problem, in particular the class averaging algorithm \cite{Singer_Zhao_Shkolnisky_Hadani:2011,Hadani_Singer:2011b,Zhao_Singer:2014}. In general, the essential goal of VDM is to integrate different kinds of local/partial information and the relationship between these pieces of local information in order to obtain the global information of the dataset; for example, the ptychographic imaging problem \cite{Marchesini_Tu_Wu:2014}, the vector nonlocal mean/median, the orientability problem \cite{Singer_Wu:2011}, etc. Numerically, the VDM depends on the spectral study of the graph connection Laplacian (GCL) \cite{Singer_Wu:2012,Chung_Zhao_Kempton:2013,Chung_Kempton:2013,Singer_Wu:2014,ElKaroui_Wu:2015a,ElKaroui_Wu:2015b}, which is a direct generalization of the graph Laplacian discussed in the spectral graph theory \cite{Fan:1996}.

\subsection{Our contribution}
A fundamental problem regarding the VDM, like that in the spectral embedding, is that if we could embed the manifold with only finite eigenvector fields; that is, can we find a finite number $n\in\NN$ so that the {\em truncated VDM} (tVDM), defined as
\begin{align}
V^{N^2}_t: M&\rightarrow \RR^{N^2}\label{Definition:tVDM}\\
x &\mapsto \left(e^{-(\lambda_i+\lambda_j)t/2}\langle X_{i}, X_{j}\rangle(x) \right)_{i,j=1}^n\,,\nonumber
\end{align}
is an embedding?
In this paper, we provide a positive answer.

\begin{restatable}[Embeddings of Riemannian Manifolds with Finite Eigenvector Fields of Connection Laplacian]{thm}{embthm}
\label{thm:emb}
For a smooth closed manifold $(M,g)$ with smooth metric and for any $t>0$, there is a positive integer $N_0$ so that the tVDM $V^{N^2}_t$ 
is a smooth embedding for all $N \geq N_0$. 
\end{restatable}

With the above theorems we could answer the raised question positively and say that with a proper chosen finite number of eigenvector fields, the tVDM is an embedding. 

Denote a class of closed smooth manifolds of dimension $d$ by
\be
\mathcal{M}_{n,\kappa, i_0,V} = \{ (M^n,g) :  |\Ric(g)| \leq \kappa ,\, \mathrm{inj(M)} \geq i_0,\,\mathrm{Vol}(M)\leq V\},
\ee
where $\mathrm{inj}(M)$ denotes the injectivity radius of $M$.

\begin{restatable}[Embeddings of Manifolds with $c^{2,\alpha}$ metrics]{thm}{rough}
Take $0<\alpha\leq1$. For $(M,g)\in\mathcal{M}_{n,\kappa,i_0,V}$ with $g \in c^{2,\alpha}$ so that $ \left| \Ric \right| \leq \kappa$, 
and for any $t>0$, there us a positive integer $N_0$ so that the tVDM $V^{N^2}_t$  is an embedding for all $N \geq N_0$.
\end{restatable}
Here $c^{k,\alpha}$ denotes the little H\"{o}lder space which is the closure of $C^{\infty}$ functions in the H\"{o}lder space $C^{k,\alpha}$. 

Throughout the paper, we try to quantify all the bounds by $n, \kappa, i_0$, and $V$ in hope that we can embedding results for the whole class of manifolds $\mathcal{M}_{n,\kappa, i_0,V}$. However, we still lack of universal estimates for Lemma \ref{conj:est:kTM} for manifolds in $\mathcal{M}_{n,\kappa, i_0,V}$. If one, can obtain such universal bounds, then the following conjecture is proved as well. 

\begin{restatable}[Embeddings of Riemannian Manifolds with Finite Eigenvector Fields of Connection Laplacian]{conj}{embconj}
\label{cor:emb}
There is a positive integer $N_0$ so that for all $(M,g)\in \mathcal{M}_{n,\kappa, i_0,V}$, $N\geq N_0$ and for any $t>0$, the tVDM $V^{N^2}_t$
is a smooth embedding. 
\end{restatable}

\subsection{Organization}
The paper is organized in the following. In Section \ref{background}, we introduce our notation convention and provide some back ground material. 

In Section \ref{sec:par}, under the assumption that the metric is smooth, we provide necessary Lemma \ref{conj:est:kTM}-\ref{lem:main} for proving the universal local parametrization for any manifold in $\mathcal{M}_{d,\kappa, i_0,V}$.

In Section \ref{sec:VDM}, under the assumption that the metric is smooth, we show the global embedding result, Theorem \ref{thm:emb}. To show the proof, we provide an immersion result in Lemma \ref{lem:local_embedding} based on Theorem \ref{thm:par}. Then, in Lemma \ref{lem:remainder bound} we control the remainder term of the series expansion of the Hilbert-Schmidt norm of the heat kernel associated with the connection Laplacian. In Section \ref{Section:RoughMetric}, we approximate a metric with low regularity by a smooth metric, and hence prove Theorem \ref{thm:emb} under the assumption that the metric is of low regularity. In Section \ref{examples}, we provide two examples to illustrate how VDM is carried out numerically.


{\bf{Acknowledgements:}}
Hau-tieng Wu's research is partially supported by Sloan Research Fellow FR-2015-65363. Chen-Yun Lin would like to thank Thomas Nyberg for his helpful discussions. 

\section{Background Material and Notation}
\label{background}

\begin{table}
\centering
\begin{tabular}{@{}l*2{>{}l}%
  l<{}l@{}}
\toprule[1.5pt]
  & \multicolumn{1}{c}{\head{Symbol}}
  & \multicolumn{1}{c}{\head{Meaning}}\\
  \cmidrule(lr){2-2}\cmidrule(lr){3-4}
  \multirow{8}{*}{Setup} 
	&  $n$ & dimension of the Riemannian manifold& \\
 	& $M$ & smooth closed Riemannian manifold & \\
	& $g$  & Riemannian metric & \\
	& $TM$ & tangent bundle of $M$ & \\
	& $\langle \cdot, \cdot \rangle$ & inner product with respect $g$ &   \\
	& $d_g( \cdot, \cdot)$ & geodesic distance & \\
	& $P_x^y$ & parallel transport from $x$ to $y$ & \\ 
	& $(\phi, U)$ & local (harmonic) coordinate chart & \\ 
  \cmidrule(lr){2-2}\cmidrule(lr){3-4}
    \multirow{4}{*}{Geometric Conditions} 
    & $V$ & volume upper bound &\\
    & $D$ & diameter upper bound &\\
    & $\kappa$ & Ricci curvature bound, $|\Ric| \leq \kappa$ &\\
    & $i_0$ & injectivity radius lower bound &\\
      \cmidrule(lr){2-2}\cmidrule(lr){3-4}
  \multirow{2}{1.1cm}{Indices} 
  	& $a,b,c,\cdots$ & indices for coordinate charts & \bfseries \\
        & $i,j,k,\cdots$ & indicies for the spectrum  & \mdseries & \\
  \cmidrule(lr){2-2}\cmidrule(lr){3-4}
 \multirow{5}{*}{Embedding} 
  	&  $N$ & dimension of the ambient Euclidean space & \\
	& $\Phi_z$ & local parametrization of $M$ & \\   
	& $\tilde{\Phi}_z$ & weighted local parametrization of $M$ & \\   
	& $\mu_{ij}$ & associated weight with respect to $X_i$ and $X_j$ & \\
	& $\Phi^{N^2}$ & embedding of $M$ into $\mathbb{R}^{N^2}$ & \\
   \cmidrule(lr){2-2}\cmidrule(lr){3-4}
  \multirow{5}{*}{Derivatives} 
  	& $\partial_i$ & partial derivative w.r.t coordinate $x_i$ & \itshape & \\
  	& $\nabla_g$ & gradient/Levi-Civita connection of $g$ & \itshape & \\
        & $\Gamma_{ab}^c$ & Christoffel symbols & \slshape & \\
        & $\Delta_g$ & Laplace-Beltrami operator & \upshape & \\
        & $\nabla_g^2$ & connection Laplacian & \upshape & \\
  \cmidrule(lr){2-2}\cmidrule(lr){3-4}
 \multirow{4}{*}{Spectra} 
 	& $\lambda_i$ & spectrum of $\nabla^2$, counting multiplicity & \itshape & \\
        & $X_i$ & $L^2$-normalized eigenvector fields, $\nabla^2 X_i = - \lambda_i X_i$ & \slshape & \\
        & $\nu_i$ & spectrum of $\Delta$, counting  multiplicity & \scshape & \\
        & $\xi_i$ & $L^2$-normalized eigenfunctions, $\Delta \xi_i = - \nu_i \xi_i$  & \upshape & \\
  \cmidrule(lr){2-2}\cmidrule(lr){3-4}
 \multirow{2}{*}{Spaces} 
        & $L^{p,\infty}(M,g)$ & the weak $L^p$ spaces & \slshape & \\
        & $ c^{2,\alpha}$& the little H\"{o}lder space & \upshape & \\
  \cmidrule(lr){2-2}\cmidrule(lr){3-4}
 \multirow{4}{*}{Heat Kernels} 
        & $k_{TM}(t,z,w)$ & heat kernel of the heat semigroup $e^{-t\nabla^2}$ & \slshape &\\
        & $k_M(t,z,w)$ & heat kernel of the heat semigroup $e^{-t\Delta}$ & \itshape & \\
        & $\nabla_v k_{TM}(t,\cdot, \cdot)$ & derivative w.r.t. the second spatial variable  & \scshape & \\
        & $\| k_{TM}(t,\cdot,\cdot) \|_{HS}^2$& Hilber-Schmidt norm of $k_{TM}(t,\cdot,\cdot)$ & \upshape & \\
  \cmidrule(lr){2-2}\cmidrule(lr){3-4}
\multirow{2}{*}{Other} 
        & $f_1 \sim_{c}^{C} f_2$ & $c f_2 \leq f_1 \leq C f_2$ & \slshape &\\
        & $c \ll C$ &  $c/C$ is sufficiently small for positive values $c, C$& \itshape & \\
\bottomrule[1.5pt]
\end{tabular}
\caption{Summary of symbols}
\end{table}

Let $(M,g)$ be a smooth compact manifold of dimension $n$ without boundary with metric $g$. We assume that $g$ is smooth except the technical part (Lemmas \ref{lem:Caccioppoli}, \ref{lem:local_L^2}, Proposition  \ref{conj:est:sup}, and Corollary \ref{cor:est:sup}) in Section \ref{sec:par} and Section \ref{Section:RoughMetric} where we assume $g$ is $c^{2,\alpha}$.
Denote $d_g(x,y)$ to be the geodesic distance between $x$ and $y$. Let $D$ denote the diameter of $(M,g)$.

Anderson  \cite{Anderson:1990} showed that the existence of harmonic coordinates on balls of uniform size is guaranteed by imposing suitable geometric conditions.
The bounds $|Ric| \leq \kappa$ and $inj(M) \geq i_0$ (or the volume bound) alone imply a lower bound on the size of balls $B$ on which one has harmonic coordinates with $C^{1,\alpha}$ bounds on the metric. Here, we restate Main Lemma 2.2 in \cite{Anderson:1990} for our case. 
\begin{lem*}
Let $(M,g)$ be a closed Riemannian $n$-manifold such that 
\be
|Ric| \leq \kappa \quad inj(M) \geq i_0
\ee
Then given any $Q >1$, $\alpha \in (0,1)$, there exists an constant $\epsilon_0 = \epsilon_0(Q, \kappa, n, \alpha)$ with the following property: given any point $z\in M$, there is a harmonic coordinate system  $\phi: U\subset \mathbb{R}^n \rightarrow B_r(z), r \geq \epsilon_0 i_0$ such that $ \phi(0) = z$, $g_{ab}(0) = \delta_{ab}$, and
\begin{eqnarray}
Q^{-1} \delta_{ab} \leq g_{ab} \leq Q \delta_{ab} \mbox{ on $U$}\\
r^{1+\alpha} \| g \|_{C^{1,\alpha}(U)} \leq Q.
\end{eqnarray}
\end{lem*}

For any $z \in M$, let $(\phi, U)$,  $\phi: U \subset \mathbb{R}^n \rightarrow M$, be a harmonic coordinate chart so that $z=\phi(0)\in \phi(U)$, $g_{ab}(0)=\delta_{ab}$, and $Q^{-1} < g < Q$ on $U$. Set
\be
R_z = \sup_r \{ r>0 : B_r(z) \subset \phi(U) \}. 
\ee

There are several advantages that we have using harmonic coordinates. First, we can obtain universal harmonic radius lower bound. Second, a metric has optimal regularity in harmonic coordinate charts. 
\begin{thm*}\cite[Theorem 2.1]{Deturck_Kazdan:1981}
If a metric $g\in C^{k,\alpha}, 1\leq k \leq \infty$ in some coordinates chart, then it is also of class $C^{k,\alpha}$ in harmonic coordinates, while it is of at least class $C^{k-2,\alpha}$ in geodesic normal coordinates.  
\end{thm*}

Last, the connection Laplacian of a vector field simplifies in harmonic coordinates and can be expressed in terms of $g$, $\partial g$, and the Ricci curvature. 
Let $\{x^a\}$ be a harmonic coordinate system on $M$ and $X = X^a \frac{\partial}{\partial x^a}$ be a vector field.

Let $\nabla_g$ denote the Levi-Civita connection of $(M,g)$ and $\nabla_g^2$ the associated connection Laplacian  \cite{CLN06} on the tangent bundle $TM$. Denote $\Delta_g$ to be the Laplace-Beltrami operator of $(M,g)$. When there is no danger of confusion, we will ignore the subscript $g$.
 Note that 
\be
\Delta x^c = g^{ab} \Gamma_{ab}^c =0,
\ee
and the Ricci curvature in the harmonic coordinates can be expressed as 
\be
- \Ric^{\,c}_d = g^{ab} \partial_a \Gamma^c_{db} + \partial_m g^{ab}  \Gamma^c_{ab} + g^{ab} \Gamma^e_{db} \Gamma_{ae}^c.
\ee
Hence,
\begin{eqnarray}
\label{eq:elliptic}
\nabla^2 \left(X^c \partial_c \right)
&=&- g^{ab}\left(\partial_a \partial_b X^c +2 \Gamma_{bd}^c \partial_a X^d + \Gamma_{ae}^c \Gamma_{bd}^e X^d + \partial_a \Gamma_{bd}^c  X^d \right) \partial_c \nonumber \\
&=&  \left(-g^{ab}\partial_a \partial_b X^c - 2g^{ab} \Gamma_{bd}^c \partial_a X^d  +\partial_d g^{ab}  \Gamma^c_{ab} X^d + Rc^{\,c}_d X^d \right) \partial_c. 
\end{eqnarray}
and the coefficients of $\nabla^2 X$ are controlled in $C^{1,\alpha}$ harmonic coordinates and the Ricci curvature bound.

It is known \cite{Gil74} that both $\nabla^2$ and $\Delta$ are self-adjoint, elliptic and that their spectra are discrete and non-positive real numbers with $-\infty$ as the only possible accumulating point. Furthermore, the eigenspaces are all finite dimensional. We denote the spectrum of $\nabla^2$ as $\left\{ -\lambda_{i}\right\} _{i=1}^{\infty}$, where  $0\leq\lambda_1 \leq \lambda_2  \leq \cdots$, counting algebraic multiplicity, and denote the corresponding orthonormal basis of eigenvector fields for $L^2(TM)$ as $\left\{ X_{i}\right\} _{i=1}^{\infty}$; that is, $\nabla^2 X_i = -\lambda_i X_i$ for all $i=1,2\cdots$ and 
\begin{equation}
\int_M \langle X_i,X_j\rangle d\sigma = \delta_{ij},
\end{equation} 
where $d\sigma$ is the Riemannian measure associated with $g$. Note that $\lambda_1$ may or may not be $0$ due to the topological constraint. Also denote the spectrum of $\Delta$ as $\left\{ -\nu_{i}\right\} _{i=1}^{\infty}$, counting algebraic multiplicity, where  $0=\nu_1 \leq \nu_2  \leq \cdots$ and denote the corresponding orthonormal basis of eigenfunctions for $L^2(M)$ as $\left\{ \xi_{i}\right\} _{i=1}^{\infty}$. Note that compared with the connection Laplacian, $\nu_1$ is always $0$. 

The heat semigroup is the family of self-adjoint operators $e^{-t\nabla^2}$, $t > 0$, with a smooth heat kernel $k_{TM}(t,x,y)$, where $k_{TM}(t,x,y)$ is smooth in $x, y \in M$ and analytic in $t$ when $t >0$ \cite{Gil74}. More precisely, for any $X\in L^2(TM)$,
\be
e^{-t\nabla^2} X(x) = \int_M k_{TM}(t,x,y) X(y) d\sigma(y).
\ee

Given an $L^2(TM)$-orthonormal basis $\{X_i\}_{i=1}^{\infty}$ of the eigenvector fields, the heat kernel of $\nabla^{2}$ can be expressed as 
\[
k_{TM}\left(t,x,y\right)=\sum_{i=1}^{\infty} e^{-\lambda_{i}t}X_{i}(x)\otimes X_{i}(y),
\]
where $t>0$ and $z,w\in M$. Its Hilbert-Schmidt norm is defined as 
\be
\label{HSnorm}
\left\Vert k_{TM}\left(t,x,y\right)\right\Vert _{HS}^{2} = \mathrm{Tr}\left( k_{TM}(t,x,y)^* k_{TM}(t,x,y) \right).
\ee
A direct computation \cite{Singer_Wu:2012} shows that the Hilbert-Schmidt norm squared of the heat kernel can be written as the series
\begin{align}
\left\Vert k_{TM}\left(t,x,y\right)\right\Vert _{HS}^{2}=\sum_{i,j}e^{-\left(\lambda_{i}+\lambda_{j}\right)t}\left\langle X_{i}(x),X_{j}(x)\right\rangle \left\langle X_{i}(y),X_{j}(y)\right\rangle. \label{Background:HeatKernelHSExpansion}
\end{align}
Based on (\ref{Background:HeatKernelHSExpansion}), the VDM (\ref{Definition:VDM}) and tVDM (\ref{Definition:tVDM}) are proposed in \cite{Singer_Wu:2012}.
Throughout the paper, we write $\nabla_v \|k_{TM}(t,\cdot,\cdot)\|_{HS}$ to denote the covariant derivative with respect to the 
second variable at time $t$.

In what follows, we use $c$ and $C$ to denote constant which may vary line by line. We write $f_1 \sim^{C}_{c} f_2$ if there exist constants $c$ and $C$ such that $c f_2 \leq f_1 \leq C f_2$. For two positive values $c,C$, we write $c\ll C$ to denote that $c/C$ is sufficiently small.

\section{Universal Local Parametrization for Smooth Manifolds}
\label{sec:par}
To embed a Riemannian manifold with finite eigenvector fields (Theorem \ref{thm:emb}), we first show in Theorem \ref{thm:par} that we could parametrize local balls of closed manifolds via eigenvector fields. 
While the proof strategy follows the ideas of \cite{Jones_Maggioni_Schul:2010}, we provide several new estimates specific for the vector fields and connection Laplacian, which have their own independent interest. 

\begin{Assumption}
\label{Theorem2:MainAssumption0} 
For $R \leq R_z$, choose  $\delta_1>0$ so that $\delta_1^2 R^2 \ll1$ and choose $\delta_0 >0$ so that $\delta_0 \ll \delta_1$. 
We consider $t$ satisfying $\frac{1}{2} \delta_1^2 R^2 \leq t \leq \delta_1^2 R^2$ and $w, z \in B_{\delta_0 R}(z) \setminus B_{\frac{1}{2}\delta_0R}(z)$.
\end{Assumption}

\begin{restatable}[Parametrization via eigenvector fields for manifolds]{thm}{parthm}
\label{thm:par}
Let $(M^n, g)$ be a smooth closed manifold with smooth metric $g$, Ricci curvature bound $\kappa$ and diameter upper bound $D$. Fix $z\in M$ and assume that 
Assumption \ref{Theorem2:MainAssumption0} holds.
For $R \leq R_z$, there exist a constant $\tau=\tau(n, Q, \kappa, i_0, \alpha, \delta_1) >1$, and $n$ pairs of indices $(i_1, j_1), \cdots, (i_n, j_n)$ so that the map 
\begin{eqnarray}
\Phi_z: B_{\tau^{-1}R}(z) &\rightarrow& \mathbb{R}^n\label{MainTheorem1:Definition:Phi}\\
x &\mapsto& (\langle X_{i_1}, X_{j_1} \rangle(x), \cdots, \langle X_{i_n}, X_{j_n} \rangle(x))\,, \nonumber
\end{eqnarray}
is a parametrization of $B_{\tau^{-1}R}(z)$, where the associated eigenvalues satisfy
\begin{equation}
\tau^{-1}R^{-2} \leq \lambda_{i_1}, \cdots, \lambda_{i_n}, \lambda_{j_1}, \cdots, \lambda_{j_n} \leq \tau R^{-2}\,.
\end{equation}
Furthermore, if $\Phi_z$ is weighted properly by
\begin{eqnarray}
\tilde{\Phi}_z: B_{\tau^{-1}R}(z) &\rightarrow& \mathbb{R}^n \\
x &\mapsto& (\mu_1\langle X_{i_1}, X_{j_1} \rangle(x), \cdots, \mu_n \langle X_{i_n}, X_{j_n} \rangle(x))\,, \nonumber
\end{eqnarray}
where
\begin{equation}
\mu_{k}:=\mu_{i_k j_k} = \left(\fint_{B_{\tau^{-1}R}(z)} \| X_{i_{k}}\|^2_g \right)^{-1/2}\left(\fint_{B_{\tau^{-1}R}(z)} \| X_{j_{k}}\|^2_g \right)^{-1/2}\,,
\end{equation}
then for any $x,y \in B_{\tau^{-1}R}(z)$, we have
\begin{equation}\label{isom_bd}
\frac{1}{\tau R} d_g(x,y) \leq \left\| \tilde{\Phi}(x) - \tilde{\Phi}(y) \right\|_{\mathbb{R}^n} \leq \frac{\tau}R d_g(x,y)\,.
\end{equation}
Here, the constants $\mu_k$, $k=1,\cdots, n$, satisfy
\begin{equation}
\mu_k \leq C,
\end{equation}
where $C$ is a constant dependent only on $n, Q, \kappa, \alpha, i_0$, and $V$. 
\end{restatable}

This Theorem indicates that the mapping $\Phi$ of a local ball is an embedding and if the weighting of the parametrization is chosen property, the embedding $\tilde{\Phi}$ is with low distortion. It further shows that the eigenvalues should not be too large or too small, which means that locally the manifold could be well parameterized by ``band-pass filtering''.

The proof strategy of Theorem \ref{thm:par} is summarized below. We provide $C^{1,\alpha}$ bounds of the eigenvector fields in Proposition \ref{conj:est:sup}. Next, we provide estimates on $\|k_{TM}(t,w,z)\|^2_{HS}$ and its gradient in Lemma \ref{conj:est:kTM}; with the bounds of eigenvector fields, we provide a control on the truncated series expansion of $\|k_{TM}(t,w,z)\|^2_{HS}$ and its gradient in Lemmas \ref{lem:trucation} and \ref{lem:modified_truncation}. In Lemma \ref{lem:main}, we show how to choose desired eigenvector fields. Last, in Theorem \ref{thm:par}, we show that the parametrization defined via those appropriately chosen eigenvector fields has the desired properties.

The key step toward the proof is the technical lemma saying that for a given eigenvector fields, locally its $C^{1,\alpha}$ norm could be well controlled by its local average $L^2$ norm. To obtain this technical lemma, we need the following Caccioppoli's type inequality.
 Note that although in this section the metric we consider is smooth, the control could be obtained when the metric is as weak as $C^{1,\alpha}$. Since the lemma has its own interest and we need the rough metric version for the eigenvector field perturbation argument later, we provide the proof under the weak assumption that the metric is $C^{1,\alpha}$. 

\begin{lem}[Caccioppoli's type inequality]
\label{lem:Caccioppoli}
Suppose that $g \in C^{1,\alpha}$ and $U$ is a bounded solution of $\nabla^2 U =0$ in $B=B_R(z)$, where $R<R_z$, with the Dirichlet boundary condition on $\partial B$. Then for $0 < r \leq R/2$,
\be
\label{est:Caccioppoli}
\|\nabla U \|_{L^2(B_r(z))} \leq C R^{-1} \| U \|_{L^2(B)}.
\ee 
for some $C=C(Q)$. 
\end{lem}

\begin{proof} 
Choose a smooth cut-off function $\psi$ on $M$ so that 
\begin{eqnarray}
& 0 \leq \psi \leq 1, \quad \|\nabla \psi\|_g \leq 2/R, \nonumber \\
& \psi \equiv 1 \mbox{ on } B_r(z), \mbox{ and } \\
& \psi \equiv 0 \mbox{ outside } B_R(z). \nonumber
\end{eqnarray}
Since $\nabla^2 U =0$, by integration by parts, we have
\be
 \int_B \psi^2 \|\nabla U\|_g^2 d\sigma=\int_B \psi^2 g^{ab}g_{cd} \nabla_a U^c \nabla_b U^d d\sigma = - 2\int_B \psi g^{ab}g_{cd} U^d \nabla_a U^c \nabla_b \psi d\sigma
\ee

By the assumption that $Q^{-1} \leq |g| \leq Q$ and the choice of $\psi$, we have 
\begin{align}
 \int_B \psi^2 \|\nabla U\|_g^2 d\sigma
\leq C \int_B \frac{2}{R} \psi \|U\|_g   \|\nabla U\|_g d\sigma
\end{align}
for some $C=C(Q)$.
Applying Young's inequality and choosing $\epsilon = \frac{1}{2 C}$, we obtain
\begin{align}
 \int_B \psi^2 \|\nabla U\|_g^2 d\sigma 
\leq&\,  \epsilon C \int_B \psi^2 \|\nabla U\|^2_g d\sigma + \frac{4 C }{\epsilon R^2} \int_B  \|U\|^2_g d\sigma \nonumber \\
\leq&\, \frac{1}{2} \int_B \psi^2 \|\nabla U\|^2_g d\sigma + \frac{8 C^2 }{ R^2} \int_B  \|U\|^2_g d\sigma\,,
\end{align}
and thus
\be
\int_{B_r(z)}  \|\nabla U\|_g^2 d\sigma 
\leq  \int_B \psi^2 \|\nabla U\|_g^2 d\sigma 
\leq \frac{16 C^2 }{ R^2} \int_B  \|U\|^2_g d\sigma
\ee
which implies the estimate (\ref{est:Caccioppoli}). 
\end{proof}

Denote $B=B_R(z)$, $R<R_z$. Let $\xi_i^B$ be the $i$-th eigenfunction of the Laplace-Beltrami operator $\Delta$ satisfying the Dirichlet boundary condition with the eigenvalue $-\nu^B_i$; that is, $\Delta^B \xi^B_i =- \nu^B_i \xi^B_i$ on $B$ and $\xi^B_i =0$ on $\partial B$. We sort the eigenvalues by $0\leq \nu^B_1 \leq \nu^B_2 \leq \cdots$ and assume that the eigenfunctions $\{\xi^B_i\}$ are $L^2(B,g)$ normalized.

\begin{lem}
\label{lem:local_L^2} 
Assume that $g\in C^{1, \alpha}$. Let $\xi^B_k$ be the Dirichlet eigenfunctions of $\Delta$ on $B$ with the eigenvalue $\nu^B_k$. Then we have the estimate
\be
\label{est:local_L^2}
\| \xi^B_k X_i \|_{L^{\frac{2n}{n-2}}(B)} \leq C \left( (\nu^B_k + \lambda_i )^{1/2}+ 2\nu^B_k\right) (\nu^B_k)^{\beta} \| X_i \|_{L^2(B)}
\ee
for some $C=C(n, Q)$, where $\beta =  \frac{n-1}{2}$ for $n$ odd and $\beta = \frac{n}{2}$ for $n$ even.
\end{lem}

\begin{proof}
By the Sobolev embedding inequality, it suffices to prove that
\be
\| \nabla (\xi^B_k X_i) \|_{L^2(B)} \leq C \left( (\nu^B_k + \lambda_i )^{1/2}+ 2\nu^B_k\right) (\nu^B_k)^{\beta} \| X_i \|_{L^2(B)}.
\ee
By a direct computation, we have 
\be
\nabla^2 (\xi^B_k X_i) =- (\nu^B_k + \lambda_i) \xi^B_k X_i + 2 g^{ab}\nabla_a \xi^B_k \nabla_b X_i
\ee
and
\be
\| \nabla (\xi^B_k X_i) \|^2_g 
\geq  \| \xi^B_k \nabla X_i\|^2_g - 2 | g^{cd} g_{ab} \xi^B_k X_i^b \nabla_d \xi^B_k \nabla_c X_i^a |.  
\ee
Combining them, we have 
\begin{align}
 \int_{B} \| \xi^B_k \nabla X_i\|^2_g d\sigma 
\leq&\, \int_{B} \| \nabla (\xi^B_k X_i) \|^2_g d\sigma +  2 \int_{B} | g^{cd} g_{ab} \xi^B_k X_i^b \nabla_d \xi^B_k \nabla_c X_i^a | d\sigma \nonumber \\
= &\, \int_{B} \langle \xi^B_k X_i, -\nabla^2(\xi^B_k X_i) \rangle_g d\sigma  +  2 \int_{B} | g^{cd} g_{ab} \xi^B_k X_i^b \nabla_d \xi^B_k \nabla_c X_i^a | d\sigma \nonumber \\
=& \,\int_{B} \langle \xi^B_k X_i,(\nu^B_k + \lambda_i) \xi^B_k X_i + 2 g^{ab}\nabla_a \xi^B_k \nabla_b X_i \rangle d\sigma\nonumber \\
&\qquad +  2 \int_{B} | g^{cd} g_{ab} \xi^B_k X_i^b \nabla_d \xi^B_k \nabla_c X_i^a | d\sigma \nonumber\\
\leq&\, (\nu^B_k + \lambda_i) \int_{B} \| \xi^B_k X_i \|^2_g d\sigma + 4\int_{B} | g^{cd} g_{ab} \xi^B_k X_i^b \nabla_d \xi^B_k \nabla_c X_i^a | d\sigma 
\end{align}
Applying the Cauchy-Schwarz inequality  and Lemma 3.5.3 in \cite{Jones_Maggioni_Schul:2010}, there exists some constant $C=C(n,Q)$ such that
\begin{align}
 \| \xi^B_k \nabla X_i\|^2_{L^2(B)}
 \leq&\, (\nu^B_k + \lambda_i) \| \xi^B_k \|^2_{L^{\infty}(B) }\|X_i \|^2_{L^2(B)} 
 + 4\| \nabla \xi^B_k\|_{L^{\infty}(B)}\|X_i\|_{L^2(B)}  \| \xi^B_k \nabla X_i\|_{L^2(B)} \nonumber \\
 \leq&\, C (\nu^B_k + \lambda_i) (\nu^B_k)^{2\beta} \|X_i \|^2_{L^2(B)} + 4(\nu^B_k)^{\beta +1} \|X_i \|_{L^2(B)} \| \xi^B_k \nabla X_i\|_{L^2(B)}
\end{align}
where $\beta = \frac{n-1}{2}$ for $n$ odd and $\beta = \frac{n}{2}$ for $n$ even.
By a direct bound, this quadratic inequality implies that 
\be
 \| \xi^B_k \nabla X_i\|^2_{L^2(B)} \leq C \left(  (\nu^B_k + \lambda_i)^{1/2} + 2\nu^B_k \right)(\nu^B_k)^{\beta}\|X_i \|^2_{L^2(B)}
\ee
for some $C= C(n,Q)$. 

Finally, we can bound $\| \nabla (\xi^B_k X_i) \|_{L^2(B)}$ by 
\begin{align}
\| \nabla (\xi^B_k X_i) \|_{L^2(B)} 
\leq&\, \| \nabla \xi^B_k \|_{L^{\infty}(B)} \|X_i \|^2_{L^2(B)}+ \| \xi^B_k \nabla X_i\|^2_{L^2(B)} \nonumber \\ 
\leq&\, C \left(  (\nu^B_k + \lambda_i)^{1/2} + 2\nu^B_k \right)(\nu^B_k)^{\beta}\|X_i \|^2_{L^2(B)}\,.
\end{align}
\end{proof}

\begin{prop}
\label{conj:est:sup}  
Suppose $g\in C^{1,\alpha}$. 
Let $P_1(x)$ and $P_2(x)$ be polynomials defined as 
\begin{equation}\label{poly}
P_1(x) = (1+x)^{\lceil \frac{n-2}{4} \rceil} \mbox{ and } P_2(x)=(1+x)^{\lceil \frac{n-2}{4} \rceil +1},
\end{equation} 
where  $\lceil x \rceil$ denotes the smallest integer not less than $x$.

For any $R \leq R_z$, we have the following for the $i$-th eigenvector field of $\nabla^2$:

\begin{equation}
\label{est:sup}
\left\Vert X_i \right\Vert _{C^0 (B_{\frac{R}{2}}(z))}
\leq C P_1(\lambda_i R^2) \left(\fint_{B_{R}(z)} \|X_i \|^2_g \right)^{1/2}\,,
\end{equation}

\begin{equation}
\label{est:Dsup}
\left\Vert \nabla X_i \right\Vert _{C^0(B_{\frac{R}{2}}(z))}\leq C \frac{1}{R}P_2 (\lambda_i R^2)\left(\fint_{B_{R}(z)}  \|X_i \|^2_g \right)^{1/2}\,,
\end{equation}

\begin{equation}
\label{est:Hsup}
\left\Vert X_i \right\Vert _{C^{\alpha}(B_{\frac{R}{2}}\left(z)\right)}
\leq C  \frac{1}{R^{\alpha}} P_1(\lambda_i R^2)\left(\fint_{B_{R}(z)} \|X_i \|^2_g\right)^{1/2}
\end{equation}

and
\begin{equation}
\label{est:DHsup}
\left\Vert \nabla X_i \right\Vert _{C^{\alpha}(B_{\frac{R}{2}}(z))}\leq C \frac{1}{R^{1+\alpha}}P_2(\lambda_i R^2)\left(\fint_{B_{R}(z)}  \|X_i \|^2_g \right)^{1/2}\,,
\end{equation}
where $C$ depends on constants $n, Q, R, \kappa$, and $\| g \|_{C^{1,\alpha}}$. 

%
\end{prop}

\begin{proof} 
The proof relies on interior estimates \cite[Lemma 4]{Dolzmann_Muller:1995} and estimates on Green's matrix in \cite[Theorem 1]{Dolzmann_Muller:1995}. Lemma 4 is proved by using the Sobolev inequality and the Schauder's estimates, which says for harmonic vector fields $\nabla^2 U =0$ on $B_r$, one has 
\be
\sup_{B_{r/2}} \|U\|_g, \sup_{B_{r/2}} \|\nabla U\|_g 
\leq C \left( \| \nabla U\|_{L^2(B_r)} + \|U \|_{L^{\frac{2n}{n-2}} (B_r)}\right)
\ee 
where $C= C(n, Q, r, \kappa, \|g\|_{C^{1,\alpha}})$.  

Theorem 1 in \cite{Dolzmann_Muller:1995} is based on the global estimates Theorems 6.4.8 and 6.5.5 in \cite{Morrey:1966} and the interior estimates. In particular, the Green's matrix on $B_r$ has the following properties
\begin{eqnarray}
G \in L^{\frac{n}{n-2},\infty}(B_r)  \mbox{ with } \| G \|_{L^{\frac{n}{n-2},\infty}(B_r)} \leq C(n, Q,  \kappa, \|g\|_{C^{1,\alpha}})\\
\nabla G \in L^{\frac{n}{n-1},\infty}(B_r)  \mbox{ with } \| \nabla G \|_{L^{\frac{n}{n-1},\infty}(B_r)} \leq C(n, Q, \kappa, \|g\|_{C^{1,\alpha}})
\end{eqnarray}
and
\begin{eqnarray}
\| G(x,w) - P_y^xG(y,w)\|_g \leq C \frac{d_g(x,y)^{\alpha}}{\max{d_g(x,w)^{2-n-\alpha}, d_g(y,w)^{2-n-\alpha}}} \\
\| DG(x,w) - P_y^xDG(y,w)\|_g \leq C \frac{d_g(x,y)^{\alpha}}{\max{d_g(x,w)^{1-n-\alpha}, d_g(y,w)^{1-n-\alpha}}} 
\end{eqnarray}
where $C = C(n, Q, r, \kappa, \|g\|_{C^{1,\alpha}})$.

Rescale $R$ to $1$ by rescaling the eigenvector field $\tilde{X}_i(y) = P^y_{\exp_z(R\exp_z^{-1}y)}X_i(\exp_z(R\exp_z^{-1}y))$, where $P^y_{\exp_z(R\exp_z^{-1}y)}$ is the parallel transport from $\exp_z(R\exp_z^{-1}y)$ to $y$. We have the following rescaling:
\begin{equation}\label{eq:rescale}
\nabla^2 \tilde{X}_i(y) = \lambda_i R^2 \tilde{X}_i(y)\,,
\end{equation}
\begin{equation}
\nabla \tilde{X}_i (y) = R P^y_{\exp_z(R\exp^{-1}_zy)}\nabla X_i (\exp_z(R\exp^{-1}_zy))\,,
\end{equation} 
and by the change of variables,
\begin{equation}
\fint_{B_1(z)} \langle \tilde{X}_i, \tilde{X}_i \rangle d\sigma= \fint_{B_R(z)} \langle X_i, X_i \rangle d\sigma. 
\end{equation}

To simplify the notation, below we use $X_i$ to represent the rescaled $\tilde{X}_i$.

Let $r_0 = 1 > r_1 > r_2 > \cdots > \frac{1}{2} = r_m$ where $m$ is to be chosen later. For $l=0,1,\ldots$, let $B_l= B_{r_l}(z)$ and let $G^{B_l}$ denote the Green's operator on $B_l$ associated with $\nabla^2$ with Dirichlet boundary condition. Write
\be 
X_i |_{B_l} = U_{(l)} +V_{(l)},
\ee
where 
\be
V_{(l)}(x) = \int_{B_l} G^{B_l}(x,y) \nabla^2 X_i(y) dy = -\lambda_i \int_{B_l} G^{B_l} (x,y) X_i (y) dy
\ee
and $\nabla^2 U_{(l)} = 0$ on $B_l$.
Note that since $g\in C^{1,\alpha}$,  $G^{B_l} \in  L^{\frac{n}{n-2}, \infty}(B_l,TM\otimes T^*M)$ and  $\nabla G^{B_l} \in L^{\frac{n}{n-1}, \infty}(B_l,TM\otimes T^*M\otimes T^*B)$ (see \cite[Theorem 1]{Dolzmann_Muller:1995}). 

We start the iteration with $l=0$. Let $p_0 = \frac{2n}{n-2}$, which is the conjugate of $2$. Let $p_1 = \frac{2n}{n-6 + \eta_1}$ for some $\eta_1, 0 < \eta_1 < 4$. 
Note that $\frac{1}{p_0} + \frac{2n-4+\eta_1}{2n} = \frac{1}{p_1} +1$ and $p_1>p_0>2$. Following the proofs of Young's inequality, we have
\begin{align} 
\| V_{(0)} \|_{L^{p_1}(B_2)}\leq&\, \| V_{(0)} \|_{L^{p_1}(B_0)} \\
=&\, \lambda_i \| \int_{B_0} G^{B_0}(x,y) X_i(y) dy\|_{L^{p_1}(B_0)} \nonumber \\
\leq&\, \lambda_i \| G^{B_0}\|_{L^{\frac{2n}{2n-4+\eta_1}}(B_0)} \| X_i\|_{L^{p_0}(B_0)}\nonumber
\end{align}
and hence
\be
\| V_{(0)}\|_{L^{p_1}(B_2)} \leq C \lambda_i  \|X_i \|_{L^{p_0}(B_0)}\,,
\ee
where $C=C(p_1,p_0,V)$, since $G^{B_0} \in L^{\frac{n}{n-2},\infty}$ and $\frac{2n}{2n-4+\eta_1}<\frac{n}{n-2}$. From the interior estimate in \cite[Lemma 4]{Dolzmann_Muller:1995} and the fact that $L^{p,\infty}\subset L^p$ for $1\leq p<\infty$, there exists a constant $C=C(n, Q, r, \kappa, \|g\|_{C^{1,\alpha}})$ such that
\begin{align}
\| U_{(0)} \|_{L^{\infty}(B_{2})} &\,\leq C(\| \nabla U_{(0)} \|_{L^2(B_1)}+\| U\|_{L^{p_0}(B_1)})\\
&\,\leq C(\| \nabla U_{(0)} \|_{L^2(B_1)}+\| U\|_{L^{p_0}(B_0)})\nonumber
\end{align}
from Lemma \ref{lem:Caccioppoli} and the fact that $p_0>2$, we have
\begin{align}
\| \nabla U_{(0)} \|_{L^2(B_1)}\leq C \|U_{(0)} \|_{L^2(B_0)}\leq C \|U_{(0)} \|_{L^{p_0}(B_0)},
\end{align}
and hence
\begin{align}
\| U_{(0)} \|_{L^{\infty}(B_{2})}& \leq C\| U_{(0)} \|_{L^{p_0}(B_0)} \label{ineq:harmonic}\\
&\leq  C(\| V_{(0)} \|_{L^{p_0}(B_0)}+\| X_i \|_{L^{p_0}(B_0)})\nonumber\\
&\leq C(\| V_{(0)} \|_{L^{p_1}(B_0)}+\| X_i \|_{L^{p_0}(B_0)})\nonumber\\
&\leq C(1+\lambda_i)\| X_i \|_{L^{p_0}(B_0)},\nonumber
\end{align}
where the third inequality holds since $p_1>p_0$. Eventually, we have
\begin{align}
\| X_i \|_{L^{p_1}(B_{2})}& \leq \| V_{(0)} \|_{L^{p_1}(B_{2})}+\| U_{(0)} \|_{L^{p_1}(B_{2})}\\
&\leq  C \lambda_i  \|X_i \|_{L^{p_0}(B_0)}+\| U_{(0)} \|_{L^{\infty}(B_2)}\nonumber\\
&\leq C(1+\lambda_i)\| X_i \|_{L^{p_0}(B_0)}.\nonumber
\end{align}

For $l=1,2,\ldots$, let $\displaystyle p_l = \frac{2n}{n-2-4l+\sum_{j=1}^l \eta_j}$ where $\eta_j$ are constants, $0 < \eta_j <4$ for all $j=1,\cdots,l$. 
Based on the same argument, by the fact that $B_{l+1}\subset B_{l}$ and $\frac{1}{p_l} + \frac{2n-4+\eta_{l+1}}{2n} = \frac{1}{p_{l+1}} +1$, we have
\begin{align}
\| V_{(l)}\|_{L^{p_{l+1}}(B_{l+2})}  &\leq \| V_{(l)}\|_{L^{p_{l+1}}(B_{l})} \nonumber\\
&\leq \lambda_i \| G^{B_l} \|_{L^{\frac{2n}{2n-4+{\color{blue}\eta_{l+1}}}}(B_l)} \| X_i \|_{L^{p_l}(B_l)} \label{ineq:eigen}\\
&\leq  C \lambda_i \| X_i \|_{L^{p_l}(B_l)}.\nonumber
\end{align}
Similarly, from the interior estimate in \cite[Lemma 4]{Dolzmann_Muller:1995} and Lemma \ref{lem:Caccioppoli}, there exists a constant $C=C(n, Q, r, \kappa, \|g\|_{C^{1,\alpha}})$ such that 
\begin{align}
\| U_{(l)} \|_{L^{\infty}(B_{l+2})} \leq C\| U_{(l)} \|_{L^{p_l}(B_{l})} \leq C(1+\lambda_i)\|X_i\|_{L^{p_l}(B_l)}
\end{align}
and hence
\begin{align}
\| X_i \|_{L^{p_{l+1}}(B_{l+1})}& \leq \| V_{(l)} \|_{L^{p_{l+1}}(B_{l+1})}+\| U_{(l)} \|_{L^{p_{l+1}}(B_{l+1})}\\
&\leq  C \lambda_i  \|X_i \|_{L^{p_l}(B_l)}+\| U_{(l)} \|_{L^{\infty}(B_{l+1})}\nonumber\\
&\leq C(1+\lambda_i)\| X_i \|_{L^{p_l}(B_l)}\,.\nonumber
\end{align}
By the iteration, we have
\begin{align}
\| X_i \|_{L^{p_{l+1}}(B_{l+2})}& \leq C(1+\lambda_i)^{l+1}\| X_i \|_{L^{p_0}(B_0)}\,.
\end{align}

Let $m$ to be the smallest integer greater or equal to $\frac{n-2}{2}$. We may choose constants $\eta_1, \cdots, \eta_{m}$ so that $p_{m} = \infty$. Then we have
\begin{align}
\| X_i \|_{L^{\infty}(B_{r_m}(z))}\leq&\, C (1 + \lambda_i)^{m}\| X_i \|_{L^{p_0}(B)}. \label{Proof:KeyLemma:supNormControl}
\end{align}

To control $\| X_i \|_{L^{p_0}(B)}$, we let $\psi = \sum_{i=1}^K a_i \xi^B_j$ be a finite sum of Dirichlet eigenfunctions on $B$ such that 
\be
\frac{1}{2} \leq \psi (x) \leq 2, \mbox{ for } x\in B 
\ee
and $\sum_{i=1}^K |a_i|\leq C, \nu^B_i \leq C$, for all $1\leq i \leq K$. 
By Lemma \ref{lem:local_L^2},
\begin{align}
\| X_i \|_{L^{p_0}(B)} 
\leq&\, 2 \| \psi X_i\|_{L^{\frac{2n}{n-2}}(B)} \\
\leq&\, 2 \sum_{k=1}^K |a_k| \| \xi^B_k X_i \|_{L^{\frac{2n}{n-2}}(B)} \nonumber\\
\leq&\, 2 \sum_{k=1}^K |a_k| \left( (\nu^B_k + \lambda_i )^{1/2}+ 2\nu^B_k\right) (\nu^B_k)^{\beta} \| X_i \|_{L^2(B)}\nonumber\\
\leq&\, C (\lambda_i+1)^{1/2} \|X_i\|_{L^2(B)}\, \nonumber
\end{align}
which, when combined with (\ref{Proof:KeyLemma:supNormControl}), implies the estimate (\ref{est:sup}). 

To bound $\| \nabla X_i \|_{L^{\infty}(B_{\frac{1}{2}}(z))} = \| \nabla X_i \|_{L^{\infty}(B_m)}$, we follow the same line by noting that 
\begin{align}
\label{ineq:grad_eigen}
\| \nabla V_{(m-1)}\|_{L^{\infty}(B_m)} 
=&\, \left\| \nabla \int G^{B_{1/2}(z)}(x,y) \nabla^2 X_i(y) d\sigma \right\|_{L^{\infty}(B_{\frac{1}{2}}(z))} \\
\leq&\, \lambda_i \| \nabla G^{B_{1/2}}\|_{L^1(B_{1/2})} \| X_i \|_{L^{\infty}(B_{\frac{1}{2}}(z))} \nonumber\\
\leq&\, C  \lambda_i (1+\lambda_i)^m \| X_i \|_{L^2(B)} \nonumber\,.
\end{align}
Again by the interior estimate in \cite[Lemma 4]{Dolzmann_Muller:1995} and Lemma \ref{lem:Caccioppoli}, there exists a constant $C=C(n, Q, r, \kappa, \|g\|_{C^{1,\alpha}})$ such that
\be
\label{ineq:grad_har}
\| \nabla U_{(m-1)}\|_{L^{\infty}(B_m)} \leq C \| U_{(m-1)} \|_{L^{p_0}(B_{m-1})}.
\ee
The bound (\ref{est:Dsup}) follows from (\ref{ineq:grad_eigen}) and (\ref{ineq:grad_har}). 

Finally, to control the H\"older seminorm in (\ref{est:Hsup}) and (\ref{est:DHsup}), we need we first choose a cut-off function $\eta$ so that $ 0 \leq \eta \leq 1$, and
\begin{eqnarray}
&\eta = 1  \mbox{ on } B_{1/2}(z) \nonumber  \\
&\eta = 0  \mbox{ outside of }  B_{3/4}(z)  \nonumber \\
&|\nabla \eta|, |\Delta \eta| \leq 4. \nonumber 
\end{eqnarray}
Jones-Maggioni-Schul  provided an example of such a cut-off function in \cite[p. 162]{Jones_Maggioni_Schul:2010}. 

Let $B_r = B_r(z)$ and $B=B_1$. Let $G$ be the Green's operator on $B_{3/4}$ with the Dirichlet boundary condition. Then
\begin{eqnarray}
\label{holder}
\left| X_i(x) - P_y^x X_i(y) \right| 
&\leq& \int_{B_{3/4}(z)} \left( G(x,w) - P_y^x G(y,w) \right) \left( \nabla^2 ( \eta X_i) \right)(w) d\sigma(w) \nonumber \\
&\leq& \left\|\nabla^2 (\eta X_i) \right\|_{L^{\infty}(B)} \int_B \left|\left( G(x,w) - G(y,w) \right)\right|  d\sigma(w)
\end{eqnarray}

Note that 
\begin{eqnarray}
\label{cutoff}
&&\left\| \nabla^2 (\eta X_i ) \right\|_{L^{\infty}(B_{3/4})} \nonumber \\
&&\leq \| \eta \nabla^2 X_i\|_{L^{\infty}(B_{3/4})} + \|(\Delta \eta) X_i\|_{L^{\infty}(B_{3/4})} + 2\|\nabla \eta \|_{L^{\infty}(B_{3/4})} \|\nabla X_i \|_{L^{\infty}(B_{3/4})}  \nonumber \\
&&\leq \left( \| \Delta \eta \|_{L^{\infty}(B)} +  \lambda_i \|\eta\|_{L^{\infty}(B)}\right) \| X_i \|_{L^{\infty}(B)} + \| \nabla \eta \|_{L^{\infty}(B)} \| \nabla X_i \|_{L^{\infty}(B)} \\
&&\leq 4 \left( (1 + \lambda_i) P_1(\lambda_i)  + P_2(\lambda_i) \right)\|X_i \|_{L^2(B)}\nonumber 
\end{eqnarray}
where the last inequality follows from the bounds (\ref{est:sup}) and (\ref{est:Dsup}).

On the other hand, by the pointwise estimates of the Green't matrix in Theorem 1 in \cite{Dolzmann_Muller:1995}, we have, for $|v| =k$,
\be
\label{est:Green}
|\nabla^v G(x,w) -P^x_y \nabla^v G(y,w)| 
\leq C \frac{d_g(x,y)^{\alpha}}{\max\{d_g(x,w)^{2-n-k-\alpha}, d_g(y,w)^{2-n-k-\alpha}\} }.
\ee

Combining (\ref{holder}), (\ref{cutoff}), and (\ref{est:Green}), we have 
\be
\left| X_i(x) - P_y^x X_i(y) \right| \leq C P_2(\lambda_i) d_g(x,y)^{\alpha} \|X_i \|_{L^2(B)}
\ee
which implies (\ref{est:Hsup}) after rescaling $1$ back to $R$. 

Similary, since 
\begin{eqnarray}
\label{D_holder}
\left| \nabla X_i(x) - P_y^x \nabla X_i(y) \right| 
&\leq& \int_{B_{3/4}(z)} \left( \nabla G(x,w) - P_y^x \nabla G(y,w) \right) \left( \nabla^2 ( \eta X_i) \right)(w) d\sigma(w) \nonumber \\
&\leq& \left\|\nabla^2 (\eta X_i) \right\|_{L^{\infty}(B)} \int_B \left|\left( G(x,w) - G(y,w) \right)\right|  d\sigma(w)
\end{eqnarray}
the bound (\ref{est:DHsup}) follows from (\ref{D_holder}), (\ref{cutoff}), and (\ref{est:Green}).
\end{proof}

\begin{cor}
\label{cor:est:sup} 
Let $\alpha$, $P_1(x)$ and $P_2(x)$ be defined as in Proposition \ref{conj:est:sup}.
For $R \leq R_z$, $x,y \in B_{\frac{R}{2}}(z)$, there exists a constant $C=C(n, Q, r, \kappa, \|g\|_{C^{1,\alpha}})$ such that the following estimates hold 
\begin{equation}\label{est:dot_sup}
\left| \langle X_i, X_j \rangle (x) \right| \leq C P_1\left(\lambda_i R^2 \right) P_1\left(\lambda_j R^2 \right)\left( \fint_{B_R(z)}  \| X_i \|^2_g \right)^{1/2} \left( \fint_{B_R(z)}   \| X_j \|^2_g \right)^{1/2},
\end{equation}

\begin{equation}\label{est:Ddot_sup}
\| \nabla \langle X_i, X_j \rangle (x)  \|_g 
\leq C \frac{1}{R}P_2 \left( \lambda_i R^2\right)P_2\left( \lambda_j R^2 \right) \left( \fint_{B_R(z)}  \| X_i \|^2_g \right)^{1/2} \left( \fint_{B_R(z)}   \| X_j \|^2_g \right)^{1/2},
\end{equation}
and
\begin{align}
\label{est:Ddot_diff}
&\| \nabla \langle X_i, X_j \rangle (x) - P_y^x\nabla \langle X_i, X_j \rangle (y) \|_g \\
\leq&\, C\frac{d_g(x,y)^{\alpha}}{R^{1+\alpha}} P_2 \left( \lambda_i R^2\right)P_2\left( \lambda_j R^2 \right) \left( \fint_{B_R(z)}  \| X_i \|^2_g \right)^{1/2} \left( \fint_{B_R(z)}   \| X_j \|^2_g \right)^{1/2}. \nonumber 
\end{align}
\end{cor}

\begin{proof}
The inequality (\ref{est:dot_sup}) follows by the Cauchy-Schwarz inequality and (\ref{est:sup}) in Proposition \ref{conj:est:sup}:
\begin{align}
 \left| \langle X_i, X_j \rangle (x) \right|  \leq&\, \| X_i \|_{C^0(B_{\frac{R}{2}}(z))} \| X_j \|_{C^0(B_{\frac{R}{2}}(z))} \nonumber \\
\leq&\, C P_1\left(\lambda_i R^2 \right) P_1\left(\lambda_j R^2 \right)\left( \fint_{B_R(z)}  \| X_i \|^2_g \right)^{1/2} \left( \fint_{B_R(z)}   \| X_j \|^2_g \right)^{1/2}. 
\end{align}
Similarly, using (\ref{est:sup}), (\ref{est:Dsup}), and the Cauchy-Schwarz inequality, we have
\begin{align}
\| \nabla \langle X_i, X_j \rangle (x) \|_g \nonumber  
\leq&\, \| \nabla X_i \|_{C^0(B_{\frac{R}{2}}(z))} \| X_j \|_{C^0(B_{\frac{R}{2}}(z))} + \| \nabla X_j \|_{C^0(B_{\frac{R}{2}}(z))} \|X_i \|_{C^0(B_{\frac{R}{2}}(z))}\\ 
\leq&\, C\frac{1}{R} P_2(\lambda_iR^2)P_2(\lambda_jR^2) \left( \fint_{B_R(z)}  \| X_i \|^2_g \right)^{1/2} \left( \fint_{B_R(z)}   \| X_j \|^2_g \right)^{1/2},
\end{align}
which gives (\ref{est:Ddot_sup}).
Last, we prove the estimate (\ref{est:Ddot_diff}). Note that since 
\begin{align}
&\langle \nabla X_i, X_j \rangle (x)-\langle \nabla X_i, X_j \rangle (y)\\
=\,&\langle \nabla X_i(x), X_j(x)-P_y^xX_j(y) \rangle+\langle P_x^y\nabla X_i(x)-\nabla X_i(y), X_j (y)\rangle, \nonumber
\end{align}
we have
\begin{align}
& \| \nabla \langle X_i, X_j \rangle (x) - \nabla \langle X_i, X_j \rangle (y) \|_g \\
\leq&\,  \| \nabla X_i \|_g (x) \| X_j(x) - P_y^xX_j(y) \|_g
	+ \| X_j \|_g (y) \|  \nabla X_i(x) - P_y^x \nabla X_i (y) \|_g\nonumber \\
& \quad +  \| X_i \|_g(x) \|  \nabla X_j(x) - P_y^x \nabla X_j(y) \|_g
	+ \| \nabla X_j \|_g(y) \| X_i(x) - P_y^xX_i(y) \|_g\,. \nonumber
\end{align}
Then (\ref{est:Ddot_diff}) follows by Proposition \ref{conj:est:sup} that 
\begin{align}
&\| \nabla \langle X_i, X_j \rangle (x) - \nabla \langle X_i, X_j \rangle (y) \|_g  \\
\leq&\, C \frac{ d_g(x,y)^{\alpha}}{R^{1+\alpha}} P_2 \left( \lambda_i R^2\right)P_2\left( \lambda_j R^2 \right) \left( \fint_{B_R(z)} \| X_i \|^2_g \right)^{1/2} \left( \fint_{B_R(z)}  \| X_j \|^2_g  \right)^{1/2}.\nonumber
\end{align}
\end{proof}


Next, we consider the heat kernel truncation approximation.
Let 
\be
\Lambda_L(A) = \{ i: \lambda_i \leq At^{-1} \} \mbox{ and } \Lambda_H(A') = \{ i: \lambda_i > A't^{-1} \},
\ee 
where $A,A'$ are chosen positive numbers. Intuitively, $\Lambda_L(A)$ includes all ``low frequency'' eigenvector fields while $\Lambda_H(A')$ includes all ``high frequency'' eigenvector fields.

\begin{lem} 
\label{conj:est:kTM}
Let $M$ be a smooth closed manifold with a smooth metric $g$. Under Assumption \ref{Theorem2:MainAssumption0}, we have the following expansions:
\begin{align}
\label{est:kTM}
\| k_{TM}(t,w,z) \|_{HS}^2 
= \frac{1}{(4\pi t)^n} (n + O(t))\left( 1 - \frac{d_g(z,w)}{2t} + O\left( \frac{d_g(z,w)^4}{t^2}\right)\right) \,.
\end{align}
For $v \in T_wM$ a unit vector parallel to $\exp_w^{-1}( z)$, 
\begin{align}
\label{est:grad_kTM}
\left| \nabla_v \| k_{TM}(t,w,z) \|_{HS}^2 \right|
= \frac{1}{(4\pi t)^n} (n + O(t))\left( \frac{d_g(z,w)}{2t} + O\left( \frac{d_g(z,w)^3}{t^2}\right)\right) \,.
\end{align}
\end{lem}

\begin{proof}
By the assumption, $d_g(z,w)^2 \ll t$. The estimate (\ref{est:kTM}) and \ref{est:grad_kTM} follow from the asymptotic expansion (see \cite[p. 1094]{Singer_Wu:2012} and \cite[p. 87]{BGV04})

\begin{equation*}
\left\Vert k_{TM}\left(t,w,z\right)\right\Vert _{HS}^{2}=\left(n+O\left(t\right)\right)\left(4\pi t\right)^{-n}\left(1-\frac{d_g(z,w)^{2}}{2t}+O\left(\frac{d_g(z,w) ^{4}}{t^{2}}\right)\right)
\end{equation*}
and its gradient follows from straightforward computation
\begin{equation*}
|\nabla_v \left\Vert k_{TM}\left(t,w,z\right)\right\Vert _{HS}^{2} |
= \left(n+O\left(t\right)\right)\left(4\pi t\right)^{-n}\left(\frac{d_g(z,w)}{2t}+O\left(\frac{d_g(z,w)^{3}}{t^{2}}\right)\right)
\end{equation*}
which gives estimate (\ref{est:grad_kTM}). 
\end{proof}

\begin{lem} 
\label{lem:trucation}
Under Assumption \ref{Theorem2:MainAssumption0}, for sufficiently large $A=A( n, Q, \kappa, \delta_0, \delta_1)>1$ and sufficiently small $A'=A'(n, Q, \kappa, \|g\|_{C^{1,\alpha}}, \delta_1)<1$, 
we can control the tail of the heat kernel. More precisely, there exist constants $c=c(A, A', n, Q, \kappa, \|g\|_{C^{1,\alpha}}, \delta_0, \delta_1 )$ and $C=C(A, A' ,n, Q, \kappa, \|g\|_{C^{1,\alpha}}, \delta_0, \delta_1)$ so that 
\begin{equation}  \label{truncated_kTM}
\left\Vert k_{TM}\left(t,w,z\right)\right\Vert _{HS}^{2}
\sim^C_c \sum_{i,j\in\Lambda_{L}\left(A\right)}e^{-\left(\lambda_{i}+\lambda_{j}\right)t}\left\langle X_{i}(z),X_{j}(z)\right\rangle \left\langle X_{i}(w),X_{j}(w)\right\rangle,
\end{equation}
and 
\begin{equation} 
\left\| \nabla \left\Vert k_{TM}\left(t,w,z\right)\right\Vert _{HS}^{2} \right\|_g
\sim^C_c
\label{truncated_grad_kTM}
\left\|\sum_{i,j\in\Lambda_{L}\left(A\right) \cap L_H(A')}e^{-\left(\lambda_{i}+\lambda_{j}\right)t}\left\langle X_{i}(w),X_{j}(w)\right\rangle \nabla \left\langle X_{i}(z),X_{j}(z)\right\rangle \right\|_g\,,
\end{equation}
where $c \to 1$ as $A'\to 0$ and $C\to 1$ as $A\to \infty$.
\end{lem}

\begin{proof}
First, note that by the Cauchy-Schwartz inequality, we have
\begin{eqnarray}
\left\Vert k_{TM}\left(t,w,z\right)\right\Vert _{HS}^{2}  \leq  \left\Vert k_{TM}\left(t,z,z\right)\right\Vert _{HS}\left\Vert k_{TM}\left(t,w,w\right)\right\Vert _{HS},\label{ineq:CS}
\end{eqnarray}
and hence we can bound the tail of the series by
\begin{align}
  & \left| \sum_{\lambda_{i}>At^{-1} \mbox{ or } \lambda_{j}>At^{-1}}e^{-\left(\lambda_{i}+\lambda_{j}\right)t}\left\langle X_{i}(z),X_{j}(z)\right\rangle \left\langle X_{i}(w),X_{j}(w)\right\rangle \right|\\
 \leq&\, e^{-A}\sum_{\lambda_{i}>At^{-1} \mbox{ or } \lambda_{j}>At^{-1}}e^{-\left(\lambda_{i}+\lambda_{j}\right)\frac{t}{2}}\left|\left\langle X_{i}(z),X_{j}(z)\right\rangle \left\langle X_{i}(w),X_{j}(w)\right\rangle \right|\nonumber\\
 \leq&\, e^{-A}\left\Vert k_{TM}(\frac{t}{2},z,z)\right\Vert _{HS}\left\Vert k_{TM}(\frac{t}{2},w,w)\right\Vert _{HS}\nonumber\\
 \leq&\, e^{-A}  k_{M}(\frac{t}{2},z,z)  k_{M}(\frac{t}{2},w,w)\,,\nonumber
\end{align}
where the last inequality holds since $\|k_{TM}(t,x,x)\|^2_{HS}\leq k_M(t,x,x)$ for all $t>0$ and $x\in M$ \cite[p.137]{Berard:1986}. Recall the upper bound the the heat kernel estimate under our manifold assumption \cite[Corollary 3.1]{Li_Yau:1986}
\begin{equation}
k_M(t,x,y) \leq \frac{C(\epsilon)^\alpha}{\sqrt{|B_{\sqrt{t}}(x)||B_{\sqrt{t}}(y)|}}\exp\left\{-\frac{d_g^2(x,y)}{(4+\epsilon)t}+\frac{C(n)\epsilon\kappa t}{\alpha-1}\right\}\,,
\end{equation} 
for any $1< \alpha<2$, $0<\epsilon<1$ and $C(\epsilon)\to \infty$ as $\epsilon\to 0$. By taking $\epsilon=1/2$ and $\alpha=3/2$, we clearly have that 
\begin{equation}
k_M(t,x,x)\leq Ct^{-n/2}, \label{bd:KernelLB:exp_decay}
\end{equation} 
where $C=C(n, Q, \kappa)$.
Thus, we have 
\[
\left|{\displaystyle \sum_{\lambda_{i}>At^{-1} \mbox{ or } \lambda_{j}>At^{-1}}e^{-\left(\lambda_{i}+\lambda_{j}\right)t}\left\langle X_{i}(z),X_{j}(z)\right\rangle \left\langle X_{i}(w),X_{j}(w)\right\rangle }\right|\leq C(n, Q, \kappa) e^{-A}t^{-n}
\]
which implies (\ref{truncated_kTM}) by choosing $A$ large enough. 

For the gradient, note that 
\be
\nabla \left\Vert k_{TM}\left(t,w,z\right)\right\Vert _{HS}^{2}=\sum_{i,j}e^{-\left(\lambda_{i}+\lambda_{j}\right)t}\left\langle X_i(w) ,X_j(w) \right\rangle  \nabla\left\langle X_{i}(z),X_{j}(z)\right\rangle\,.
\ee 
We first consider the contribution of the high frequency part; that is, when $\lambda_{i}$ or $\lambda_{j}$ is large enough. By a direct bound, we have
\begin{align}
 &\left\| {\displaystyle \sum_{\lambda_{i}>At^{-1} \mbox{ or } \lambda_{j}>At^{-1}}e^{-\left(\lambda_{i}+\lambda_{j}\right)t}\langle X_i(w) ,X_j(w) \rangle  \nabla\langle X_{i}(z),X_{j}(z)\rangle}\right \|_g  \\
  \leq&\,  2  \sum_{\{i,j: \lambda_{i}>At^{-1}\}}e^{-\left(\lambda_{i}+\lambda_{j}\right)t} \left| \langle X_i(w),X_j(w) \rangle \right| \left\| \nabla\langle X_{i}(z),X_{j}(z)\rangle\right \|_g\,.\nonumber
\end{align}
Since $\lambda_{i}\geq At^{-1}$ and $\lambda_{j}\geq\lambda_{j}/2$, we have
\begin{align}
\label{eq:3-1}
 & \left\| \sum_{\lambda_{i}>At^{-1} \mbox{ or } \lambda_{j}>At^{-1}}e^{-\left(\lambda_{i}+\lambda_{j}\right)t}\langle X_i (w),X_j (w)\rangle  \nabla\langle X_{i}(z),X_{j}(z)\rangle\right \|_g  \\
  \leq&\, 2e^{-\frac{A}{2}}\sum_{\{i,j:\,\lambda_{i}>At^{-1}\}}e^{-\left(\lambda_{i}+\lambda_{j}\right)\frac{t}{2}}\left|\langle X_i(w),X_j(w)\rangle\right| \| \nabla\langle X_i(z),X_j(z)\rangle  \|_g \nonumber \\
   \leq&\, 2e^{-\frac{A}{2}} \left(\sum_{\{i,j:\,\lambda_{i}>At^{-1}\}}e^{-\left(\lambda_{i}+\lambda_{j}\right)\frac{t}{2}}\langle X_i(w) ,X_j(w) \rangle^2 \right)^{1/2}\nonumber \\ 
  & \qquad \qquad \times \left(\sum_{\{i,j:\,\lambda_{i}>At^{-1}\}}e^{-\left(\lambda_{i}+\lambda_{j}\right)\frac{t}{2}}\|\nabla\langle X_i(z),X_j(z)\rangle  \|_g^{2}\right)^{1/2}  \,,\nonumber
\end{align}
where the last inequality follows from the Cauchy-Schwartz inequality.

To control $\sum_{\{i,j:\,\lambda_{i}>At^{-1}\}}e^{-\left(\lambda_{i}+\lambda_{j}\right)\frac{t}{2}}\|\nabla\langle X_i(z),X_j(z)\rangle  \|_g^{2}$, we need the following bounds (\ref{eq:3-2}), (\ref{bd:exp_decay}), and (\ref{ineq:L1}). By (\ref{est:Ddot_sup}) in Corollary \ref{cor:est:sup}, we have
\begin{align}
 \label{eq:3-2}
 & e^{-\left(\lambda_{i}+\lambda_{j}\right)\frac{t}{2}}\|\nabla\left\langle X_{i}(z),X_{j}(z)\right\rangle \|_g^{2}\\
\leq &\,   e^{-\left(\lambda_{i}+\lambda_{j}\right)\frac{t}{2}} \frac{C}{R^2}P_2\left(\lambda_{i}R^2\right)^2P_2\left(\lambda_{j}R^2\right)^2\fint_{B_{R}(z)} \| X_{i}\|_g^2 \fint_{B_{R}(z)} \|X_j \|_g^2 \nonumber \,,
\end{align}
where $C=C(n, Q, \kappa, \| g\|_{C^{1,\alpha}})$. Since $e^{-\left(\lambda_{i}+\lambda_{j}\right)\frac{t}{4}}$ decays exponentially as $i$ increases and $P_2\left(\lambda_{i}R^{2}\right)$ increases polynomially, by the choice of $A$ and $\delta_0$ so that $\delta_0^2 R^2/4< t$, we have
\begin{align}
\label{bd:exp_decay}
e^{-\left(\lambda_{i}+\lambda_{j}\right)\frac{t}{4}}P_2\left(\lambda_{i}R^{2}\right)^2P_2\left(\lambda_{j}R^{2}\right)^2
\leq e^{-\left(\lambda_{i}+\lambda_{j}\right)\frac{t}{4}}P_2\left(\lambda_{i}\frac{4t}{\delta_0^2}\right)^2P_2\left(\lambda_{j}\frac{4t}{\delta_0^2}\right)^2 \,, 
\end{align}
which is bounded by a constant depending on $t$ and $\delta_0$. 

Furthermore, by the Cauchy-Schwarz inequality and Kato's inequality, we have
\begin{align}\label{ineq:L1}
&\sum_{i,j}e^{-\left(\lambda_{i}+\lambda_{j}\right)\frac{t}{4}} \|X_i(x)\|^2_g\|X_j(y)\|^2_g \\
=& \,\left( \sum_i e^{-\frac{\lambda_i t}{4}} \|X_i(x)\|_g^2\right)\left( \sum_j e^{-\frac{\lambda_j t}{4}} \|X_j(y)\|_g^2\right)\nonumber  \\
\leq&\, \left( \sum_i e^{-\frac{\lambda_i t}{4}}\right) \left( \sum_i e^{-\frac{\lambda_i t}{4}} \|X_i(x)\|_g^4\right)^{1/2}  \left( \sum_j e^{-\frac{\lambda_j t}{4}} \|X_j(y)\|_g^4\right)^{1/2} \nonumber \\
\leq&\, n \left(\int_M k_M(\frac{t}{4},x,x) d\sigma\right) k_{M}(\frac{t}{8},x,x)k_{M}(\frac{t}{8},y,y)  \nonumber\,,
\end{align}
where $k_M(t,\cdot,\cdot)$ denotes the heat kernel of the Laplace-Beltrami operator. To be more precise, in (\ref{ineq:L1}), we apply
\begin{equation}
\sum_i e^{-\frac{\lambda_i t}{4}}\leq n\sum_i e^{-\frac{\nu_i t}{4}}=n \int_M k_M(\frac{t}{4},x,x) d\sigma
\end{equation} 
by Kato's inequality, where $\nu_i$ are eigenvalues of the Laplace-Beltrami operator, and 
\begin{align}
\sum_i e^{-\frac{\lambda_i t}{4}} \|X_i(x)\|_g^4\leq&\, \sum_{i,j} e^{-\frac{(\lambda_i+\lambda_j) t}{8}} \langle X_i(x),X_j(x)\rangle^2\\
=&\,\|k_{TM}(\frac{t}{8},x,x)\|^2_{HS}\leq k_M(\frac{t}{8},x,x)\,,\nonumber
\end{align} 
where the first inequality holds since $\langle X_i(x),X_j(x)\rangle^2\geq 0$ for all $i,j$ and the last inequality holds, again, due to the fact that $\|k_{TM}(t,x,x)\|^2_{HS}\leq k_M(t,x,x)$ for all $t>0$ and $x\in M$ (see \cite[p.137]{Berard:1986}.)


Using the bounds (\ref{eq:3-1}), (\ref{eq:3-2}), (\ref{bd:exp_decay}), (\ref{ineq:L1}) and (\ref{bd:KernelLB:exp_decay}), 
we can bound the contribution of the high frequency part
\begin{align}
 & \left\|\sum_{\lambda_{i}>At^{-1} \mbox{ or } \lambda_{j}>At^{-1}}e^{-\left(\lambda_{i}+\lambda_{j}\right)t}\left\langle X_{i},X_{j}\right\rangle (w) \nabla\left\langle X_{i},X_{j}\right\rangle (z)\right\|_g\\
 \leq&\, \frac{C}{R} t^{-\frac{3n}{4}} e^{-\frac{A}{2}}
 \left( \fint_{B_R(z)}k_{M}(\frac{t}{8},x,x) d\sigma \fint_{B_R(z)}  k_{M}(\frac{t}{8},y,y)  d\sigma \right)^{1/2}\nonumber \\
\leq &\,\frac{C}{R} e^{-\frac{A}{2}} t^{-\frac{5n}{4}} \leq \frac{C}{\delta_0^{5/2}R^{7/2}} e^{-\frac{A}{2}},\nonumber
\end{align}
where the last inequality holds due to the choice of $\delta_0R_z\ll t^{1/2}$, which is arbitrarily small provided $A$ is sufficiently large.

To bound the contribution of the low frequency part for the gradient, we proceed as below: \begin{align}
&\left\| \sum_{i\notin \Lambda_H(A') \mbox{ or }j\notin \Lambda_H(A')}e^{-\left(\lambda_{i}+\lambda_{j}\right)t}\left\langle X_{i}(w),X_{j}(w)\right\rangle  \nabla\left\langle X_{i}(z),X_{j}(z)\right\rangle  \right\|_g \\
\leq&\, 2 \sum_{i\notin \Lambda_H(A'), j} e^{-\left(\lambda_{i}+\lambda_{j}\right)t}
\left|\langle X_{i}(w),X_{j}(w)\rangle  \right| \| \nabla \langle X_{i}(z),X_{j}(z)\rangle  \|_g \nonumber\\
 \leq&\,  2 \left\| k_{TM}(t,w,w)\right\|_{HS} \left(  \sum_{i\notin \Lambda_H(A'), j} e^{-\left(\lambda_{i}+\lambda_{j}\right)t}\|\nabla \langle X_{i}(z),X_{j}(z)\rangle \|_g^2 \right)^{1/2}\,, \nonumber
\end{align}
where the last inequality holds by the Cauchy-Schwartz inequality.
We can further bound the last term by the following:  
\begin{align}
&\sum_{i\notin \Lambda_H(A'), j} e^{-\left(\lambda_{i}+\lambda_{j}\right)t}\|\nabla \langle X_{i}(z),X_{j}(z)\rangle \|_g^2 \\
\leq&\, C \sum_{i\notin \Lambda_H(A'), j} e^{-\left(\lambda_{i}+\lambda_{j}\right)t} \frac{P_2(\lambda_i\frac{4t}{\delta_1^2})^2P_2(\lambda_j\frac{4t}{\delta_1^2})^2}{R^2} \fint_{B_R(z)}\| X_i\|^2_g\fint_{B_R(z)}\| X_j\|^2_g  \nonumber \\
 \leq &\,\frac{C}{R^2} P_2(\frac{4A'}{\delta_1^2})\left( \sum_{i\notin \Lambda_H(A')} e^{-\lambda_i t}  \fint_{B_R(z)}\| X_i\|^2_g \right) \left(\sum_j e^{-\lambda_j t} P_2(\lambda_j\frac{4t}{\delta_1^2})^2 \fint_{B_R(z)}\| X_j\|^2_g\right) \nonumber \\
 \leq&\, \frac{C}{R^2} \left(\int_M k_M(\frac{t}{2},x,x)\right)^{1/2} \left( \fint_{B_R(z)}\| k_{TM}(\frac{t}{4},x,x)\|_{HS}\right)^2 \left(\sum_{\lambda_i < A't^{-1}} e^{-\lambda_i t}\right)^{1/2} \nonumber \,,
\end{align}
where the first inequality holds due to (\ref{est:Ddot_sup}) in Corollary \ref{cor:est:sup} and the last inequality holds by similar arguments as in (\ref{eq:3-2}), (\ref{bd:exp_decay}), and (\ref{ineq:L1}).

By Weyl's law for the connection Laplacian (see \cite[p.92]{BGV04} or \cite[Lemma 4.2]{Wu:2014}) and the Kato's type inequality (see \cite[p.135]{Berard:1986}),
\begin{align}
\sum_{\lambda_i \leq A't^{-1}}  e^{-\lambda_i t}
\leq e \sum_{\lambda_i \leq A't^{-1}} e^{-\frac{\lambda_i t}{A'}} 
\leq e n \int_M k_M(\frac{t}{A'}, x,x) d\sigma
\leq e n (A't^{-1})^{n/2} 
\end{align}
where the last inequality follows from Proposition 3.1.2 in \cite{Jones_Maggioni_Schul:2010}. 
Therefore, 
\begin{align}
&\left\|\sum_{\lambda_{i}<A't^{-1}\mbox{or }\lambda_{j}<A't^{-1}}e^{-\left(\lambda_{i}+\lambda_{j}\right)t}\left\langle X_{i},X_{j}\right\rangle (w) \nabla\left\langle X_{i},X_{j}\right\rangle (z)\right\|_g\\
 \leq&\, \frac{C}{R}   t^{-3n/2} A'^{\frac{n}{4}}\leq \frac{C}{\delta_0^{3n}R^{3n+1}}  A'^{\frac{n}{4}}\,,\nonumber
\end{align}
which implies (\ref{truncated_grad_kTM}) if $A'$ is sufficiently small.
\end{proof}

To finish the heat kernel truncation approximation, we further restrict ourselves on the subset of eigenvector fields with large enough gradients. 
For any pair of eigenvector fields $X_i, X_j$, define the associated weight $\mu_{ij}$ as 
\be
\mu_{ij}:=\left(\fint_{B_{\delta_0R}(z)}\left\langle X_{i},X_{i}\right\rangle \right)^{-1/2}\left(\fint_{B_{\delta_0R}(z)}\left\langle X_{j},X_{j}\right\rangle \right)^{-1/2}\,.
\ee
which depends on $\delta_0R$.
For $v\in T_zM$ of unit length and $c_0>0$, define
\begin{align}
\Lambda_{E}(v,z,R,\delta_{0},c_{0})\label{Lambda_E} :=\,\left\{ \lambda_i, \lambda_j :\, |\nabla_{v}\left\langle X_{i},X_{j}\right\rangle (z)| \geq \frac{c_0}{R}\mu_{ij}^{-1}\right\}.
\end{align}
Note that for a given $v$, when $c_0$ is chosen small enough, $\Lambda_{E}$ and $\Lambda_{L}\left(A\right) \cap \Lambda_{H}(A') \cap \Lambda_E(v,z,R,\delta_{0},c_{0})$ are not empty.
The truncated series in Lemma \ref{lem:trucation} can be slightly modified as the following so that the condition of $\Lambda_{E}(v,z,R,\delta_{0},c_{0})$ is included. 

\begin{lem}
\label{lem:modified_truncation} 
Under Assumption \ref{Theorem2:MainAssumption0}, let $A$ and $A'$ be chosen as in Lemma \ref{lem:trucation}. 
For $c_0 = c_0(n, Q, \kappa, \|g\|_{C^{1,\alpha}}, \delta_1, \delta_1)$ small enough, there exist $c=c(n,Q, \kappa, \|g\|_{C^{1,\alpha}}, \delta_0, \delta_1, A, A', c_0)$ and $C=C(n,Q, \kappa, \|g\|_{C^{1,\alpha}}, \delta_0, \delta_1, A, A', c_0)$ so that
 for $v\in T_zM$ of unit length parallel to $\exp_z^{-1}(z)$, we have
\begin{equation}\label{truncatedE_grad_kTM}
|\nabla_{v}\left\Vert k_{TM}\left(t,w,z\right)\right\Vert _{HS}^{2}|
\sim^C_c
|\sum_{i,j\in\Lambda}e^{-\left(\lambda_{i}+\lambda_{j}\right)t} \nabla_{v}\left\langle X_{i}(z),X_{j}(z)\right\rangle \left\langle X_{i}(w),X_{j}(w)\right\rangle |\,,
\end{equation}
where $\Lambda := \Lambda_{L}\left(A\right) \cap \Lambda_{H}(A') \cap \Lambda_E(v,z,R_z,\delta_0,c_0)$ and $c, C\to 1$  when $c_0\to 0$.
\end{lem}

\begin{proof}
Let $\Lambda^c = \Lambda_L(A) \cap \Lambda_H(A') \cap \left( \Lambda_E(v,z,R_z,\delta_0,c_0)\right)^c$. We have
\begin{align}
&| \sum_{i,j \in \Lambda^c} e^{-\left(\lambda_{i}+\lambda_{j}\right)t}  \nabla_v\langle X_i(z), X_j(z) \rangle  \langle X_i(w), X_j(w) \rangle| \nonumber \\
\leq&\, \left( \sum_{i,j \in \Lambda^c}  e^{-\left(\lambda_{i}+\lambda_{j}\right)t} \left|  \langle X_i(w), X_j(w) \rangle\right|^2 \right)^{1/2} \left( \sum_{i,j \in \Lambda^c} e^{-\left(\lambda_{i}+\lambda_{j}\right)t} |\nabla_v \langle X_i(z), X_j(z) \rangle|^2\right)^{1/2} \nonumber\\
 \leq&\, k_{M}(t,w,w)^{1/2} \left( \sum_{i,j \in \Lambda^c} e^{-\left(\lambda_{i}+\lambda_{j}\right)t} \frac{c_0^2}{R^2} \fint_B\langle X_i, X_i\rangle \fint_B \langle X_j, X_j\rangle\right)^{1/2}\nonumber\\
 \leq&\, \frac{c_0}{R} k_{M}(t,z,z)^{1/2} \left(  \int_M k_M(t,x,x) \right)^{1/2} \fint_Bk_{M}(t,w',w')^2 , \nonumber
\end{align}
where the first inequality holds due to the Cauchy-Schwartz inequality, the second inequality holds by the fact that 
\begin{equation}
\sum_{i,j \in \Lambda^c}  e^{-\left(\lambda_{i}+\lambda_{j}\right)t} \left|  \langle X_i(w), X_j(w) \rangle\right|^2\leq \Vert k_{TM}(t,w,w)\Vert_{HS}\leq  k_{M}(t,w,w)
\end{equation} 
and 
\begin{equation}
\sum_{i,j \in \Lambda^c} e^{-\left(\lambda_{i}+\lambda_{j}\right)t} |\nabla_v \langle X_i(z), X_j(z) \rangle |\leq \sum_{i,j \in \Lambda^c} e^{-\left(\lambda_{i}+\lambda_{j}\right)t} \frac{c_0^2}{R^2} \fint_B\langle X_i, X_i\rangle \fint_B \langle X_j, X_j\rangle
\end{equation} 
by the constraint (\ref{Lambda_E}) of $\Lambda_E(v,z,R_z,\delta_0,c_0)$, and the last inequality follows from the Kato's type inequality and similar arguments as in (\ref{ineq:L1}). This inequality proves (\ref{truncatedE_grad_kTM}) by reducing $c_0$. 
\end{proof}

 \begin{lem} 
 \label{lem:main} Let $A, A'$ and $c_0$ be chosen as in  Lemma \ref{lem:trucation} and Lemma \ref{lem:modified_truncation} under Assumption \ref{Theorem2:MainAssumption0}. For any pair $\lambda_{i},\lambda_{j}\in\Lambda_{E} \left(v,z,R,\delta_{0},c_{0}\right)$, there exist constants $C_1=C_1(c_0)$, $C_2=C_2(n, Q, \kappa, \|g\|_{C^{1,\alpha}}, c_0)$, and $\tau = \tau(n, Q, \kappa, \|g\|_{C^{1,\alpha}}, c_0 )$ independent of $i,j$ so that  
\begin{equation}
\label{est:grad}
C_1 R^{-1} \mu_{ij}^{-1} \leq |\nabla_{P^{z'}_zv}\left\langle X_{i},X_{j}\right\rangle (z')|\leq C_2 R^{-1} \mu_{ij}^{-1}.
\end{equation}

Moreover, for sufficiently large $A$ and sufficiently small $A'$ chosen in Lemma \ref{lem:trucation}, there exist $\lambda_{i},\lambda_{j}\in\Lambda_{L}\left(A\right)\cap \Lambda_H(A') \cap \Lambda_{E}\left(v,z,R,\delta_{0},c_{0}\right)$
so that 
\be \label{unif_bd}
\mu_{ij}\leq C\left(n,\kappa,D, \delta_0, \delta_1\right)\,. 
\ee
\end{lem}

\begin{proof}
The upper bound in (\ref{est:grad}) follows straightforward from inequalities (\ref{est:Ddot_sup}); the lower bound in (\ref{est:grad}) comes directly from (\ref{est:Ddot_diff}) and the definition of $\Lambda_E(v,z,R,\delta_0, c_0)$ in (\ref{Lambda_E}). Indeed, we have
\begin{align}
| \nabla_{P_z^{z'}v}\left\langle X_{i},X_{j}\right\rangle (z')|\geq&\, |\nabla_{v}\left\langle X_{i},X_{j}\right\rangle (z)|- |\nabla_{P_z^{z'}v}\left\langle X_{i},X_{j}\right\rangle (z')-\nabla_{v}\left\langle X_{i},X_{j}\right\rangle (z)| \nonumber\\
\geq&\,  \frac{c_0}{R}\mu_{ij}^{-1}-\frac{Cd_g(z,z')^\alpha}{R^{1+\alpha}}\mu_{ij}^{-1},
\end{align}
which leads to the lower bound if $\tau$ is chosen small enough.

Now we show (\ref{unif_bd}). 
Denote $\Lambda:=\Lambda_{L}\left(A\right)\cap \Lambda_H(A') \cap \Lambda_{E}\left(v,z,R,\delta_{0},c_{0}\right)$ and $B:=B_{\delta_0R}(z)$. From previous Lemmas and (\ref{est:grad}), we obtain
\begin{align}
 & | \nabla_v \Vert k_{TM}(t,w,z) \Vert^2_{HS} |  \\
 \sim^C_c&\, \left| \sum_{\lambda_i,\lambda_j \in \Lambda} 
	e^{-(\lambda_i+\lambda_j)t} \langle X_i(w),X_j(w) \rangle \nabla_v \langle X_i(z),X_j(z) \rangle\right|  \nonumber\\
 \leq&\, C_2 R^{-1}  \sum_{\lambda_i,\lambda_j \in \Lambda} 
	e^{-(\lambda_i+\lambda_j)t} \left|  \langle X_i(w),X_j(w) \rangle  \right| \mu_{ij}^{-1}\,,\nonumber
 \end{align} 
where the approximation is by (\ref{truncated_grad_kTM}) and the inequality is by (\ref{est:grad}), which by the Cauchy-Schwarz inequality is bounded by
\begin{align}
 C_2  R^{-1} \left( \sum_{\lambda_i,\lambda_j \in \Lambda}  e^{-2(\lambda_i+\lambda_j)t} \langle X_i(w),X_j(w) \rangle^2 \right)^{1/2} \left( \sum_{\lambda_i,\lambda_j \in \Lambda} \mu_{ij}^{-2}\right)^{1/2}.
\end{align}
Hence,
\begin{align}
\label{ineq:kTM_grad_kTM}
| \nabla_v \Vert k_{TM}(t,z,w) \Vert^2_{HS} |\leq&\, C_2 R^{-1} \Vert k_{TM}(2t,w,w) \Vert_{HS} \left( \sum_{\lambda_i,\lambda_j \in \Lambda} \mu_{ij}^{-2} \right)^{1/2}\\
\leq &\,C_2 R^{-1}  k_{M}(2t,w,w) \left( \sum_{\lambda_i,\lambda_j \in \Lambda} \mu_{ij}^{-2} \right)^{1/2}\,.\nonumber
\end{align}
The bound (\ref{ineq:kTM_grad_kTM}), together with the estimate (\ref{est:grad_kTM}), imply that
\begin{equation*}
t^{-n} \frac{R}{t} \leq C R^{-1} t^{-n/2}  \left( \sum_{\lambda_i,\lambda_j \in \Lambda} \mu_{ij}^{-2} \right)^{1/2}
\end{equation*}
and thus,
\begin{equation}
\label{ineq:main_lemma}
t^{-n}\left( \frac{R^2}{t}\right)^2 \leq  C \sum_{\lambda_i,\lambda_j \in \Lambda} \mu_{ij}^{-2}.
\end{equation}

On the other hand, it is known that  \cite[Lemma 4.2]{Wu:2014}  there exists a constant $B(n,\kappa,D)$ dependent only on $n, \kappa$ and $D$ such that for $T>0$, the following inequality holds:
\begin{equation}\label{ineq:weyl_bd}
\# \{j : 0 \leq \lambda_j \leq T \} \leq en + B(n,\kappa,D) T^{n/2}. 
\end{equation}
Here $e$ is the Euler's number. 

Since $\lambda_i, \lambda_j \leq At^{-1}$, it follows from (\ref{ineq:main_lemma}), (\ref{ineq:weyl_bd})
that there exist some $i,j$ such that
\begin{equation}
t^{-n}\left( \frac{R^2}{t}\right)^2 \leq  C \left( en + B(n,\kappa,D) (At^{-1})^{n/2}  \right)^2 \mu_{ij}^{-2} 
\end{equation}
That is, 
\be
\mu_{ij} \leq C\left( en + B(n,\kappa,D) (At^{-1})^{n/2}  \right) t^{n/2} \frac{t}{R^2} <C(n, Q,\kappa, \|g\|_{C^{1,\alpha}}, D, \delta_0,\delta_1)
\ee
since $t\leq\delta_1 R_z$. Thus, we have the uniform bound (\ref{unif_bd}). Note that the diameter upper bound can be controlled by
the volume upper bound and injectivity radius lower bound. Hence, 
\be
\mu_{ij} \leq C(n, Q, \kappa, V, i_0, \delta_1).
\ee
\end{proof}


With the above Lemmas, we are now ready to prove Theorem \ref{thm:par}. 
\parthm*

\begin{proof}

Take $t,\delta_0$ in Lemma \ref{lem:trucation} and $c_0$ chosen in Lemma \ref{lem:modified_truncation}. To simplify the notation, denote $\mu_l:=\mu_{i_l,j_l}$, where $l=1,\ldots,d$. Pick an arbitrary unit vector $v_1$. It follows from Lemma \ref{lem:main} that there exist $i_1,j_1 \in \Lambda_L(A) \cap \Lambda(A') \cap \Lambda_E(v_1,z,R,\delta_0,c_0)$ such that $\mu_{1}\leq C\left(d,\kappa,D, \delta_0, \delta_1, A, A', g\right)$ and
\begin{equation}
\left| \mu_1 \nabla_{v_1} \langle X_{i_1}, X_{j_1}  \rangle(z) \right| \geq \frac{C_1}{R}\,. 
\end{equation}
Let $v_2$ be a unit vector orthogonal to $\nabla \langle X_{i_1}, X_{j_1} \rangle (z)$. Applying Lemma \ref{lem:main} again, we find another pair of indices $i_2, j_2 \in \Lambda_L(A) \cap \Lambda_H(A') \cap \Lambda_E(v_2,z, R,\delta_0,c_0)$ so that $\mu_{2}\leq C\left(n,\kappa,D, \delta_0, \delta_1, A, A', g\right)$ and
\begin{equation}
\left| \mu_2 \nabla_{v_2} \langle X_{i_2}, X_{j_2}  \rangle(z) \right| \geq \frac{C_1}{R}\,.
\end{equation}
Note that by the choice of $v_2$, we have $\nabla_{v_2} \langle X_{i_1}, X_{j_1} \rangle(z)=0$, and hence $(i_1,j_1)\neq (i_2,j_2)$. We could proceed and find $v_3$ and so on.
Suppose that we have chosen $k$ vectors $v_1, v_2, \cdots, v_k$, $k<n$, and the corresponding indices $(i_1,j_1), (i_2,j_2), \cdots, (i_k,j_k)$ such that 
\begin{equation}
\left| \mu_l \nabla_{v_l} \langle X_{i_l}, X_{j_l}  \rangle(z) \right| \geq \frac{C_1}{R}\quad \mbox{for all }  l=1,\cdots,k.
\end{equation}
Pick another unit vector $v_{k+1}$ so that $v_{k+1}$ is orthogonal to $\{ \nabla  \langle X_{i_1}, X_{j_1}  \rangle(z)\}_{l=1}^k$. By Lemma \ref{lem:main}, we have $i_{k+1},j_{k+1}  \in \Lambda_L(A) \cap \Lambda_H(A') \cap \Lambda_E(v_{k+1},z, R,\delta_0,c_0)$  so that $\mu_{k+1}\leq C\left(n,\kappa,D, \delta_0, \delta_1, A, A', g\right)$ and
\begin{equation}
\left| \mu_{k+1} \nabla_{v_{k+1}} \langle X_{i_{k+1}}, X_{j_{k+1}}  \rangle(z) \right| \geq \frac{C_1}{R}.
\end{equation}
We now claim $\{v_1, v_2, \cdots, v_{k+1} \}$ is linearly independent.
To show this, we assume that
\begin{equation}\label{Proof:MainTheorem1:Independent}
a^1 v_1 + a^2 v_2 + \cdots + a^{k+1} v_{k+1} = 0.
\end{equation}
Consider the $(k+1)\times (k+1)$ matrix
\begin{equation}
A_{k+1} := \left( \mu_m \nabla_{v_n} \langle X_{i_m}, X_{j_m}  \rangle(z) \right)_{m,n=1,\cdots,k+1}.
\end{equation}
On one hand, by (\ref{Proof:MainTheorem1:Independent}), we have 
\begin{align}
 A_{k+1} \left( \begin{matrix}
 a^1\\
 a^2\\
\vdots \\
a^{k+1}
\end{matrix} \right) 
=&\, a^1 \left( 
\begin{matrix}
 \mu_1 \nabla_{v_1} \langle X_{i_1}, X_{j_1}  \rangle \\
  \vdots \\
 \mu_{k+1} \nabla_{v_1} \langle X_{i_{k+1}}, X_{j_{k+1}}  \rangle 
\end{matrix}
\right) 
+ \cdots + 
 a^{k+1} \left( 
\begin{matrix}
 \mu_1 \nabla_{v_{k+1} } \langle X_{i_1}, X_{j_1}  \rangle \\
  \vdots  \\
 \mu_{k+1} \nabla_{v_{k+1} } \langle X_{i_{k+1}}, X_{j_{k+1}}  \rangle  
\end{matrix}
\right)\nonumber \\
=&\,  \left(
\begin{matrix}
\langle a^1v_1+\cdots + a^{k+1} v_{k+1}, \mu_1 \nabla \langle X_{i_1}, X_{j_1} \rangle(z)  \rangle \\
\vdots \\
\langle a^1v_1+\cdots + a^{k+1} v_{k+1}, \mu_{k+1} \nabla \langle X_{i_{k+1}}, X_{j_{k+1}} \rangle(z)  \rangle
\end{matrix}
\right) =0\,.
\end{align}
On the other hand, the matrix $A_{k+1}$ is lower-triangular since $\nabla_{v_n} \langle X_{i_m}, X_{j_m}\rangle =0$, for all $n>m$, which follows by the choice of the vectors $\{v_l\}_{l=1}^{k+1}$. Therefore $a^l =0$ for all $l=1,\cdots, k+1$ and hence $\{v_1, v_2, \cdots, v_{k+1} \}$ is linearly independent. 

With the above chosen eigenvector fields, we now show the embedding propoerty. For any $x_1,x_2 \in B_{\delta_0 R}(z)$, let $\gamma(t), 0\leq t\leq 1$ be the geodesic curve joining $x_1=\gamma(0)$ and $x_2=\gamma(1)$. 
Under the harmonic coordinate chart $(\phi, U), \phi: U \subset \mathbb{R}^n \rightarrow M$ so that $\phi(0)=z$ and $g_{ab}(0)=\delta_{ab}$, we have $\gamma(t) = \phi(\tilde{\gamma}(t))$ for some curve $\tilde{\gamma}(t) \subset U$. We also assume that 
$\|\tilde{\gamma}'(t)\|_g= d_g(x_1,x_2)$. 
Using the basis $\{v_l \}_{l=1}^n \subset T_zM \cong \mathbb{R}^n$, we have
\be \label{local_gamma}
\tilde{\gamma}'(t) = \sum_{l=1}^n a^l(t) v_l \quad \mbox{and} \quad \gamma'(t) = \sum_{l=1}^n a^l(t) d\phi|_{\tilde{\gamma}(t)}(v_l)
\ee

We claim that there exists $c>0$ so that
\begin{equation}\label{bd_1}
\left\|  \nabla \tilde{\Phi} \circ \phi |_{0}  \tilde{\gamma}'(t) \right\|_{\mathbb{R}^d} \geq \frac{c}{R} \left\|  \tilde{\gamma}'(t) \right\|_g =\frac{c}{R},\end{equation}
that is, the geodesic deformation has a lower bound. We prove this statement by contradiction. 
Denote
 $\tilde{\Phi}_k := \left( \mu_1 \langle X_{i_1}, X_{j_1} \rangle, \cdots, \mu_k \langle X_{i_k}, X_{j_k} \rangle \right)$ so that $\tilde{\Phi} = \tilde{\Phi}_d$.
Suppose that for all $ k=1, \cdots, n$,
\begin{equation}
\label{ass_contradiction}
\left\|  \nabla \tilde{\Phi}_k \circ \phi |_0   \frac{d}{dt} \tilde{\gamma}(t) \right\|_{\mathbb{R}^k} \leq \frac{c_m}{R}, 
\end{equation}
for any small constant $c_m$. 
%
We have, by the local parametrization (\ref{local_gamma}),  
\be
\left\| \nabla \tilde{\Phi}_k \circ \phi |_{0}   \tilde{\gamma}'(t) \right\|_{\mathbb{R}^k} 
=  \left\| \sum_{l=1}^n a^l(t) \nabla  \tilde{\Phi}_k \circ \phi |_0 v_l \right\|_{\mathbb{R}^k}
= \left\| \sum_{l=1}^n a^l \nabla_{v_l}  \tilde{\Phi}_k(z) \right\|_{\mathbb{R}^k}.
\ee
Since $\nabla_{v_l} \langle X_{i_m}, X_{j_m}\rangle=0$ for all $m<l$, the first $k$ terms are the only non-zero terms, that is,
\be
\left\| \nabla \tilde{\Phi}_k \circ \phi |_{0}   \tilde{\gamma}'(t) \right\|_{\mathbb{R}^k} 
= \left\| \sum_{l=1}^k a^l \nabla_{v_l} \tilde{\Phi}_k(z) \right\|_{\mathbb{R}^k}.
\ee 
We further note that 
\begin{align}
\left\| \sum_{l=1}^k a^l \nabla_{v_l} \tilde{\Phi}_k(z) \right\|_{\mathbb{R}^k} 
\geq&\, \left| a^k  \mu_k \nabla_{v_k}\langle X_{i_k}, X_{j_k} \rangle(z) \right| -\sum_{l=1}^{k-1} \left\| a^l \nabla_{v_l} \tilde{\Phi}_k(z) \right\|_{\mathbb{R}^k} \\
\geq&\, \frac{c_0}{R}|a^k| - \frac{c}{R} \sum_{l=1}^{k-1} | a^l|\nonumber
\end{align}
since, by the choice of eigenvectors fields, $\mu_k \nabla_{v_k}\langle X_{i_k}, X_{j_k} \rangle(z) \geq \frac{c_0}{R}$ and, by (\ref{est:Ddot_sup}) in Corollary \ref{lem:trucation}, $\mu_{ij}|\nabla_{v_l}\langle X_i, X_j \rangle(z)| \leq \frac{c}{R}$.  
By the assumption (\ref{ass_contradiction}), we have
 \be
\label{ass_contradiction2}
 \frac{c_0}{R}|a^k| - \frac{c}{R} \sum_{l=1}^{k-1} | a^l| \leq \left\|  \nabla \tilde{\Phi}_k \circ \phi |_0   \frac{d}{dt} \tilde{\gamma}(t) \right\|_{\mathbb{R}^k} \leq \frac{c_m}{R}\,.
 \ee
By induction, we now show that
\begin{equation}
\label{induction}
\left| a^k \right| \leq \frac{c_m}{c_0}\big(\frac{c}{c_0}+1\big)^{k-1}.
\end{equation}
For $k=1$, (\ref{induction}) follows from (\ref{ass_contradiction2}). 
Suppose that (\ref{induction}) is true for $k=l$. Then 
\begin{align}
|a^{l+1} |\leq&\, \frac{c_m}{c_0} + \frac{c}{c_0}  \sum_{i=1}^{l} | a^i| \\
\leq&\, \frac{c_m}{c_0} + \frac{c}{c_0} \sum_{i=1}^{l} \frac{c_m}{c_0}\big(\frac{c}{c_0}+1\big)^{i-1} \nonumber\\
=&\, \frac{c_m}{c_0}\left(1+ \frac{c}{c_0} \frac{(\frac{c}{c_0}+1)^l-1}{\frac{c}{c_0}+1-1}\right) =\frac{c_m}{c_0}\big(\frac{c}{c_0}+1\big)^l\nonumber
\end{align}
and we have (\ref{induction}).
The inequality (\ref{induction}) implies $\left| a^l\right|$ is arbitrarily small for all $l=1,\cdots, n$ for $c_m$ sufficiently small,  which leads to a contradiction since $\|\tilde{\gamma}'(t)\|_g = d_g(x_1,x_2)$ is constant.

By Corollary \ref{cor:est:sup}, we also have 
\begin{equation}\label{bd_2}
\Vert \nabla \tilde{\Phi} |_x - \nabla \tilde{\Phi} |_z \Vert_{op} \leq C \left( \frac{d_g(x,z)}{R}\right)^{\alpha} \frac{1}{R} \leq C \delta_0^{\alpha} \frac{1}{R} 
\end{equation}
which is small when $\delta_0$ is small enough. 

With the above preparation, the lower bound in (\ref{isom_bd}) follows by the bounds (\ref{bd_1}) and (\ref{bd_2}). Indeed, we have 
\begin{align}
&\left| \tilde{\Phi}(x_2) - \tilde{\Phi}(x_1)\right| \\
=&\, \left| \int_0^1 \nabla \left(\tilde{\Phi} \circ \phi \right)_{\tilde{\gamma}(t)} \frac{d}{dt}\tilde{\gamma}(t)dt\right| \nonumber\\
=&\, \left| \int_0^1 \nabla \tilde{\Phi} \circ \phi |_0  \frac{d}{dt}\tilde{\gamma}(t) + \left( \nabla \tilde{\Phi} \circ \phi|_{\tilde{\gamma}(t)} - \nabla \tilde{\Phi} \circ \phi |_0\right)  \frac{d}{dt}\tilde{\gamma}(t)dt \right| \nonumber\\
\geq &\, \int_0^1 \frac{c}{R} \|\tilde{\gamma}'(t)\|_g dt = \frac{c}{R} d_M (x_1, x_2),\nonumber
\end{align}
where the last inequality holds due to (\ref{bd_1}) and (\ref{bd_2}) when $\delta_0$ is small enough.

To finish the proof of the low distortion property (\ref{isom_bd}) of the mapping $\tilde{\Phi}_z$, we should the upper bound in (\ref{isom_bd}). Note that b Corollary \ref{cor:est:sup}, for all $k,l=1, \cdots, n$,
\begin{equation}
\left|  \mu_k \nabla_{v_l} \langle X_{i_k}, X_{j_k}\rangle |_{\gamma(t)} \right| \leq  \frac{C_2}{R},
\end{equation}
which leads to the upper bound in (\ref{isom_bd}),
\be
\left\| \tilde{\Phi}(x_2) - \tilde{\Phi}(x_1) \right\|_{\mathbb{R}^n} 
= \left\|  \int_0^1 \nabla \tilde{\Phi} |_{\gamma(t)}  \gamma'(t) dt  \right\|_{\mathbb{R}^n}
\leq  \frac{C_2}{R}d_g(x_1, x_2)\,.
\ee
\end{proof}

\section{Manifold Embedding by Truncated Vector Diffusion Map}
\label{sec:VDM}

The goal in this section is to prove Theorem \ref{thm:emb} under the assumption of the smooth metric; that is, we could embed manifolds in $\mathcal{M}_{n,\kappa, i_0, V}$ into $\mathbb{R}^{N^2}$ for some finite $N$, where $N$ depends only on $n,\kappa$, $i_0$ and $V$.  
To find $n$ and the embedding for $\mathcal{M}_{n, \kappa, i_0, V}$, we need two lemmas. First, in Lemma \ref{lem:local_embedding}, we show that based on Theorem \ref{thm:emb}, for any $M \in \mathcal{M}_{n,\kappa, i_0,V}$, $M$ can be immersed into an Euclidean space. Note that we have to link the conditions of $\mathcal{M}_{n,\kappa, i_0,V}$ back to the assumptions for Theorem \ref{thm:emb}. Second, in Lemma \ref{lem:remainder bound} we control the remainder term of the series expansion of $\| k_{TM} (t,w,z) \|^2_{HS}$; together with the upper bound of $\| k_{TM} (t,w,z) \|^2_{HS}$, the immersion is improved to the embedding and we finish the proof of Theorem \ref{thm:emb}. 

\begin{lem}
\label{lem:local_embedding}
For any $M \in \mathcal{M}_{n,\kappa,i_0,V}$ and any $z \in M$, there exist a positive integer $m$ depending only on $n,\kappa,i_0,V$ and $\tau$  and a constant $\epsilon >0$ so that 
\begin{eqnarray}
\Phi_z^{m^2}: B_\epsilon(z) &\rightarrow& \mathbb{R}^{m^2} \\
x &\mapsto& (\langle X_i(x), X_j(x) \rangle)_{i,j=1}^m\nonumber
\end{eqnarray}
is a smooth embedding. 
\end{lem} 

\begin{proof}
We control the eigenvalues by applying the part (a) of Lemma 4.2 in \cite{Wu:2014}. To do so, note that for each $M\in \mathcal{M}_{n,\kappa, i_0, V}$, since $\mathrm{Vol}(M)$ is bounded by $V$ from above and the injectivity is bounded by $i_0$ from below, by the packing argument, the diameter of $M$ is bounded from above \cite[p. 267]{Anderson_Cheeger:1992}. Thus, by the part (a) of Lemma 4.2 in \cite{Wu:2014}, there exists a positive constant $A(n,\kappa, i_0, V)$ so that
\be
\label{ev_lower_bound}
\lambda_j \geq A(n,\kappa, i_0, V) j^{2/n}.
\ee

Choose $m$ be the smallest integer so that 
\be
A(n, \kappa, i_0,V) (m+1)^{2/n} > \tau r_h^{-2},
\ee 
where $\tau=\tau(n,Q)$ is the constant chosen in Theorem \ref{thm:par}.
Therefore, following (\ref{ev_lower_bound}), 
\be
\lambda_{m+1} \geq A(n, \kappa, i_0,V) (m+1)^{2/n}> \tau r_h^{-2},
\ee
 for any $M \in \mathcal{M}_{n,\kappa,i_0, V}$.
On the other hand, by Theorem \ref{thm:par}, for a given $M \in \mathcal{M}_{n,\kappa,i_0,V}$, there are $n$ pairs of indices $(i_1, j_1), \cdots, (i_n, j_n)$ so that the map 
\begin{eqnarray}
\Phi_z: B_{\tau^{-1}r_h}(z) &\rightarrow& \mathbb{R}^n\\
x &\mapsto& (\langle X_{i_1}, X_{j_1} \rangle(x), \cdots,  \langle X_{i_n}, X_{j_n} \rangle(x)), \nonumber
\end{eqnarray}
is an embedding, where $x\in B_{\tau^{-1}r_h}(z)$ and 
\be
\label{thm:par(b)}
\tau^{-1}r_h^{-2} \leq \lambda_{i_1}, \cdots, \lambda_{i_n}, \lambda_{j_1}, \cdots, \lambda_{j_n} \leq \tau r_h^{-2}. 
\ee
Since $m$ is a universal natural number chosen to satisfy $\lambda_{m+1}>\tau r_h^{-2}$, by  (\ref{thm:par(b)}), we conclude that the mapping $\Phi_z^{m^2} : B_{\epsilon}(z) \rightarrow \mathbb{R}^{m^2}$ is an embedding with $\epsilon = \tau^{-1} r_h$.

\end{proof}

In the next Lemma, we would control the difference between the diagonal terms of the heat kernel and its truncation.

\begin{lem}
\label{lem:remainder bound}
Take $M\in \mathcal{M}_{n,\kappa,i_0,V}$. Denote
\be
R_k(t) := \sup_{x \in M} \sum_{i\mathrm{ \,or \,} j \geq k} e^{-(\lambda_i + \lambda_j)t} \langle X_i(x), X_j(x) \rangle^2.
\ee
For all $k \in \mathbb{N}$, there is a function $E_k(t): \mathbb{R}^+ \rightarrow \mathbb{R}^+$ such that for all $M \in \mathcal{M}_{n,\kappa,i_0,V}$, 
\be
R_k(t) \leq E_k(t) t^{-(3n/2+1)}\label{Definition:E_kForRemainder}
\ee
and $\lim_{k\rightarrow \infty} E_k(t) =0$ for a fixed $t>0$. 
\end{lem}

\begin{proof}
It suffices to prove that 
\be
\sum_{i \geq k, j} e^{-(\lambda_i + \lambda_j)t} \langle X_i(x), X_j(x) \rangle^2 \leq E_k(t) t^{-n}. 
\ee

Consider the positive measure 
\be
d\mu_x(\lambda, \nu) = \sum_{i \geq k, j} \langle X_i(x), X_j(x) \rangle^2 \delta_{\lambda_i}\times \delta_{\lambda_j}
\ee
where $\delta_{\lambda_i}$ is the Dirac measure at $\lambda_i \in \mathbb{R}$.  By a direct calculation, we have
\begin{align}
e^{-(\lambda + \nu)t} \sum_{i \geq k, j} \langle X_i(x), X_j(x) \rangle^2 
=&\, \int_{\mathbb{R}} \int_{\mathbb{R}} e^{-(\lambda + \nu)t} d\mu_x(\lambda, \nu)  \nonumber\\
=&\, \int_{\mathbb{R}} \int_{\mathbb{R}} \left( \partial_{\lambda} \partial_{\nu} e^{-(\lambda + \nu)t} \right) \sum_{\lambda_k \leq \lambda_i \leq \lambda, \lambda_j \leq \nu} \langle X_i(x), X_j(x) \rangle^2 d\lambda d\nu\,,
\end{align}
where the last equality follows by the definition of the derivative in the sense of distribution. Hence, 
\begin{align}
e^{-(\lambda + \nu)t} \sum_{i \geq k, j} \langle X_i(x), X_j(x) \rangle^2 
=&\, \int_{0}^\infty \int_{\lambda_k}^{\infty} t^2  e^{-(\lambda + \nu)t}  \sum_{\lambda_k \leq \lambda_i \leq \lambda, \lambda_j \leq \nu} \langle X_i(x), X_j(x) \rangle^2 d\lambda d\nu \nonumber \\
=&\, \int_{0}^\infty \int_{t \lambda_k}^{\infty}  e^{-(\lambda' + \nu')}  \sum_{\lambda_k \leq \lambda_i \leq \frac{\lambda'}{t}, \lambda_j \leq \frac{\nu'}{t}} \langle X_i(x), X_j(x) \rangle^2 d\lambda' d\nu'\,,
\end{align}
where $\lambda' = \lambda t$ and $\nu' = \nu t$.
Since the lower bound of injectivity and the upper bound of volume imply the upper bound of the diameter, by the inequality (32) in \cite{Wu:2014}, we have 
\be
\sum_{\lambda_i \leq \frac{\lambda'}{t}, \lambda_j \leq \frac{\nu'}{t}} \langle X_i(x), X_j(x) \rangle^2 
\leq \frac{C(n,\kappa,i_0,V)}{\mathrm{Vol}(M)^2}\left( \frac{\lambda'}{t} \right)^{n/2} \left(\frac{\nu'}{t} \right)^{n/2},
\ee
where $C(n,\kappa,i_0,V)$ is a universal constant depending only on $n,\kappa,i_0$ and $V$. 

Since $\int_{0}^\infty {\nu'}^{n/2} e^{-\nu' t} d\nu' =t^{-n/2-1} \int_{0}^\infty x^{n/2} e^{-x} dx=\Gamma(n/2+1)t^{-(\frac{n}{2}+1)}$, where $\Gamma$ is the Gamma function, and $tA(n,\kappa, i_0,V) k^{2/n} \leq t\lambda_k$ by (\ref{ev_lower_bound}), we have 
\begin{align}
e^{-(\lambda + \nu)t} \sum_{i \geq k, j} \langle X_i(x), X_j(x) \rangle^2 
\leq&\, C t^{-(3n/2+1)} \int_{t\lambda_k}^{\infty} {\lambda'}^n e^{-\lambda' t} d\lambda' \\
\leq&\,  C t^{-(3n/2+1)} \int_{t A(d,\kappa, i_0,V) k^{2/d}}^{\infty} {\lambda'}^n e^{-\lambda' t} d\lambda'.\nonumber
\end{align}
Define 
\be
E_k(t) := C \int_{t A(d,\kappa, i_0,V) k^{2/n}}^{\infty} {\lambda'}^n e^{-\lambda' t} d\lambda.
\ee
Hence $R_k(t) \leq E_k(t)t^{-(3n/2+1)}$ and it is clear that $\lim_{k\rightarrow \infty} E_k(t) =0$ for a fixed $t>0$. 
\end{proof}


With the above Lemmas, we are ready to prove Theorem \ref{thm:emb}.

\embthm*

\begin{proof} 
By Lemma \ref{lem:local_embedding}, there exists a integer $m$ so that $\Phi_z^{m^2}$ is a smooth embedding for all $z\in M$ and all $M\in \mathcal{M}_{n,\kappa, i_0,V}$. Thus, if we choose $\tilde{m}\geq m$, we know that $\Phi^{\tilde{m}^2}$ is an immersion. It suffices to show that for $x,y \in M$ so that $d_g(x,y)>\epsilon$, $\Phi^{\tilde{m}^2}(x) \neq \Phi^{\tilde{m}^2}(y)$.

By the asymptotic expansion of the heat kernel (see \cite[p. 1094]{Singer_Wu:2012} and \cite[p. 87]{BGV04}), when $t\ll 1$,
\be
\| k_{TM}(t,x,y) \|_{HS}^2 = (4\pi t)^{-n} e^{-\frac{ \| v\|_g^2}{2t}}\left( 1 - \frac{\Ric(v,v)}{6} + O(\|v\|_g^3) \right)(n+O(t)),
\ee
where $v \in T_xM$ so that $y = \exp_x(v)$. Therefore, for all $M \in \mathcal{M}_{n,\kappa, i_0,V}$ and a fixed constant $\epsilon>0$, there exists a universal constant $C_{U,\epsilon}$ so that for any $x,y \in M$ such that $d_g(x,y) \geq \epsilon$,
\be
\| k_{TM}(t,x,y)\|^2_{HS} \leq \frac{C_{U,\epsilon}}{t^n} \exp \left(-\frac{\epsilon^2}{2t} \right).
\ee
Let $c_1(t)  = \inf_{x\in M}  \| k_{TM}(t,x,x)\|^2_{HS}$. Clearly, $c_1(t) = (4\pi t)^{-n}(n+O(t))>0$.

Let $\epsilon$ and $m$ be chosen as in Lemma \ref{lem:local_embedding}.
Define 
\be
G(t) := c_1(t) - \sup_{d_g(x,y)\geq \epsilon}\| k_{TM}(t,x,y)\|^2_{HS}. 
\ee 
Since $\sup_{d_g(x,y)\geq \epsilon}\| k_{TM}(t,x,y)\|^2_{HS}\leq t^{-n}\exp\left( -\frac{\epsilon^2}{2t} \right)\to 0$ as $t\to 0^+$, $G(t) \rightarrow c_1(t)$ as $t\rightarrow 0^+$. By the definition of $c_1(t)$ and Lemma \ref{conj:est:kTM}, for $t \in (0, \frac{i_0}{2}]$,
\begin{align}
G(t) &\leq \inf _{d_g(x,y)\geq \epsilon}
\| k_{TM}(t,x,x)\|^2_{HS} - \sup_{d_g(x,y)\geq \epsilon} \| k_{TM}(t,x,y)\|_{HS}^2\nonumber\\
&\leq \inf _{d_g(x,y)\geq \epsilon} 
(\| k_{TM}(t,x,x)\|^2_{HS} - \| k_{TM}(t,x,y)\|_{HS}^2)\,.\label{Proof:SecondTheorem:GtFirstBound}
\end{align}

Choose $t_0 \in (0,\frac{i_0}{2}]$ so that $G(t_0) \geq \frac{4}{5}c_1(t_0) $. Then choose $\tilde{m}\geq m $ big enough so that $R_{\tilde{m}+1}(t_0)\leq \frac{1}{5}c_1(t_0)$, where $E_{\tilde{m}+1}$ is defined in (\ref{Definition:E_kForRemainder}). Denote the truncated heat kernel by

\be
k_{TM}^{(\tilde{m})}(t,x,y):=\sum_{i=1}^{\tilde{m}}e^{-\lambda_{i}t}X_{i}\left(x\right)\otimes X_{i}\left(y\right)
\ee
so that
\begin{align}
\| k_{TM}^{(\tilde{m})}(t,x,y)\|^2_{HS} &\,=\sum_{i,j\leq \tilde{m}} e^{-(\lambda_i + \lambda_j)t}\langle X_i(x), X_j(x)\rangle \langle X_i(y), X_j(y) \rangle\\
&\,=\langle  V^{\tilde{m}^2}_t(x),V^{\tilde{m}^2}_t(y)\rangle\,.
\end{align}
Note that 
\begin{eqnarray}
&&\sup_{x,y \in M} ( \| k_{TM}(t_0,x,y)\|^2_{HS} - \| k_{TM}^{(\tilde{m})}(t_0,x,y)\|^2_{HS})  \nonumber\\
&& = \sup_{x,y \in M}  \sum_{i \mbox{ or } j \geq \tilde{m}+1} e^{-(\lambda_i + \lambda_j)t_0}\langle X_i, X_j\rangle(x) \langle X_i, X_j \rangle (y)\label{Proof:SecondTheorem:GtFirstRelationship}\\
&& \leq R_{\tilde{m}+1}(t_0)  \leq \frac{1}{5}c_1(t_0)\,, \nonumber
 \end{eqnarray}
where the inequalities hold by the Cauchy-Schwartz inequality, (\ref{Definition:E_kForRemainder}) and the chosen $\tilde{m}$. With the above, we have 
\begin{align}
\frac{4}{5}c_1(t_0) \leq&\, G(t_0) \\
\leq&\, \| k_{TM}(t_0,x,x)\|^2_{HS} - \| k_{TM}(t_0,x,y)\|_{HS}^2\nonumber \\
\leq&\, \| k_{TM}^{(\tilde{m})}(t_0,x,x)\|^2_{HS} -\| k_{TM}^{(\tilde{m})}(t_0,x,y)\|^2_{HS}  +\frac{2}{5}c_1(t_0) \,,\nonumber
\end{align}
where the second inequality comes from (\ref{Proof:SecondTheorem:GtFirstBound}) and the last inequality holds due to (\ref{Proof:SecondTheorem:GtFirstRelationship}).
Therefore  $ \| k_{TM}^{(\tilde{m})}(t_0,x,x)\|^2_{HS} \neq \| k_{TM}^{(\tilde{m})}(t_0,x,y)\|^2_{HS}$, which indicates
\be
V^{\tilde{m}^2}_{t_0}(x)\neq V^{\tilde{m}^2}_{t_0}(y).
\ee
and that tVDM is an embedding.
Since tVDM $V^{\tilde{m}^2}_t$ with any time $t>0$ and $V^{\tilde{m}^2}_{t_0}$ differ only by constant scalings, they are diffeomorphic. We thus conclude the proof of the Theorem. 

\end{proof}

\section{Universal Local Parametrization and Embedding for Rough Manifolds}
\label{Section:RoughMetric}

In this section, we finish the proof of Theorem \ref{thm:emb} when the metric is of lower regularity $g \in c^{2,\alpha}$. 
The key step is to approximate $g$ by smooth metrics $g_{\epsilon}$. The parametrization for $(M,g)$ can be approximated by the parametrizations for smooth manifolds $(M,g_{\epsilon})$ defined as in Theorem \ref{thm:par} and hence we have Theorem \ref{thm:emb} for $g \in c^{2,\alpha}$. The approximation relies on the $C^1$ bound of the eigenvector field perturbation under the $c^{2,\alpha}$ metric perturbation, which is stated in Theorem \ref{thm:ev_convergence} below. While Theorem \ref{thm:ev_convergence} is of independent interest, we include it here to complete the argument. The same proof could be applied to study the $C^1$ bound of the eigenfunction perturbation under the metric perturbation with a suitable regularity, which generalizes Theorem 21 in \cite{Berard_Besson_Gallot:1994}.

Suppose that metrics $g$ and $h$ on $M$ satisfy 
\be
(1-\epsilon) \| g \|_{c^{2,\alpha}} < \| h\|_{c^{2,\alpha}} < (1+\epsilon)  \| g \|_{c^{2,\alpha}}.\label{Equation:MetricApproximation}
\ee
By the assumption and the min-max principle, and hence the continuity of eigenvalues, there exists $\epsilon' = O(\epsilon)$ such that for any $X \in C^{\infty}(TM)$, 
 \begin{eqnarray}
 & (1-\epsilon') \| X \|^2_{L^2(g)} \leq \| X \|^2_{L^2(h)} \leq  (1+ \epsilon') \| X \|^2_{L^2(g)}  \label{Calpha:Bound:Equation1}\\
 & (1-\epsilon') \| \nabla_g X \|^2_{L^2(g)} \leq \| \nabla_h X \|^2_{L^2(h)} \leq  (1+ \epsilon') \|\nabla_g X \|^2_{L^2(g)} \label{Calpha:Bound:Equation2} \\
 & (1 - \epsilon') \lambda_j(g) \leq \lambda_j(h) \leq (1 + \epsilon') \lambda_j(g) \quad \mbox{for any } j\geq 0 \label{Calpha:Bound:Equation4}
\end{eqnarray}
and
By the same argument for the continuity of eigenvalues, we have that for any $X \in C^{\infty}(TM)$,
\begin{align}
(1-\epsilon') \| \nabla^2_g X \|^2_{L^2(g)} \leq \| \nabla^2_h X \|^2_{L^2(h)} \leq  (1+ \epsilon') \|\nabla^2_g X \|^2_{L^2(g)} \label{Calpha:Bound:Equation3}\,.
\end{align}

\begin{restatable}{thm}{evconv}
\label{thm:ev_convergence}
Let $h$ be a metric on $M$ of dimension $n$ such that $(1-\epsilon) \| g \|_{c^{2,\alpha}} < \| h\|_{c^{2,\alpha}} < (1+\epsilon)  \| g \|_{c^{2,\alpha}}$. 
Suppose for both $(M,g)$ and $(M,h)$, the infectivity radii are bounded from below by $i_0>0$, the volumes are bounded above by $V>0$, and the Ricci curvatures satisfy $|\Ric| \leq \kappa$. 
Let $k_0 \in \mathbb{N}$ be given. Then there exist constants $\eta_{g,i}(\epsilon), 1 \leq i \leq k_0$, which go to zero with $\epsilon$, such that for any orthonormal basis $\{ Y_i \}$ of eigenvector fields of $\nabla^2_h$, one can associate an orthonormal basis $\{X_i\}$ of eigenvector fields of $\nabla^2_g$ satisfying
\be
\label{sup_bd:ev}
\| X_i - Y_i \|_{L^{\infty}(B,g)} \leq \eta_{g,i,K}(\epsilon)
\ee
and 
\be
\label{sup_bd:gradev}
\| \nabla_g X_i - \nabla_g Y_i \|_{L^{\infty}(B,g)} \leq \eta_{g,i,K}(\epsilon)
\ee
for all $i <k_0$, where $B:=B_R(z)$. 
\end{restatable}

The associated basis $ X_i, 1 \leq i \leq J,$ are chosen by the Gram-Schmidt process. Lemma \ref{lem:error_approx}
 and Proposition \ref{prop:orthonoram_approx} control the norms of $\| X_i - Y_i \|_{H^1(g)}$ and $\| \nabla^2_g X_i - \nabla^2_g Y_i \|_{L^2(g)}$. To obtain the sup-norm bound (\ref{sup_bd:ev}) and (\ref{sup_bd:gradev}), we need to control $\| X_i \|_{L^{\infty}(B,g)}$, $\| Y_i \|_{L^{\infty}(B,h)}$, $\|\nabla_g X_i \|_{L^{\infty}(B,g)}$, $ \|\nabla_g Y_i \|_{L^{\infty}(B,h)}$ , $\|\nabla^2_g X_i \|_{L^{\infty}(B,g)}$, and $ \|\nabla^2_g Y_i \|_{L^{\infty}(B,h)}$ on local balls which can be proved by Proposition \ref{conj:est:sup} using  Lemma \ref{lem:local_L^2}, \ref{lem:Caccioppoli}.

Let $\Lambda_1 < \Lambda_2 < \cdots < \Lambda_k <\cdots$ be the distinct eigenvalues of $\nabla^2_g$. Let $E_k$ and $m_k$ be the corresponding eigenspaces and multiplicities, and $N_k = m_1 + m_2 + \cdots + m_k$. 
By the min-max principle, the eigenvalues is continuously perturbed with respect to the deviation (\ref{Equation:MetricApproximation}). Therefore, for a fixed $k_0$, we can find $\epsilon_0>0$ so that for all $\epsilon<\epsilon_0$ and $h$ satisfies (\ref{Equation:MetricApproximation}), the first $N_{k_0}$ eigenvalues of $\nabla^2_h$, where we count the multiplicity, are all less than $\Lambda_{k_0+1}$ and contained in pairwise disjoint intervals $I_k$, for $1\leq k\leq k_0$.
Let $F_k, 1\leq k \leq k_0$, be the sum of eigenspaces of $\nabla^2_h$ corresponding to the eigenvalues $\lambda_j(h)$ contained in the interval $I_k$ about $\Lambda_k$. Let $\pi_k$ denote the orthogonal projection in $L^2(g)$ onto the eigenspace $E_k$.

\begin{lem}
\label{lem:error_approx}
Fix $k_0$. For $k \leq k_0$, there exist functions $\alpha_k(\epsilon)$, $\beta_k(\epsilon)$, and $\gamma_k(\epsilon)$ which go to zero with $\epsilon$ such that 
\begin{itemize}
\item[(a)] $\| (\pi_k - Id) Y \|_{H^1(g)} \leq \alpha_k(\epsilon) \|Y\|_{L^2(h)}$ \quad \mbox{for any } $Y\in F_k$,
\item[(b)] For any $X \in \oplus_{i=1}^k E_i$ and for any $Y$ that is $L^2(h)$-orthogonal to $\oplus_{i=1}^k F_i$,
\be
\left| \langle X,Y \rangle_g \right| \leq \beta_k(\epsilon) \| X \|_{L^2(g)} \|Y\|_{L^2(h)}.
\ee  
\item[(c)] $\| \nabla^2_g (\pi_k - Id) Y \|_{L^2(g)} \leq \gamma_k(\epsilon) \|Y\|_{L^2(h)}$ \quad \mbox{for any } $Y\in F_k$.

\end{itemize}
\end{lem}

\begin{proof} 
The proof of assertions (a) and (b) follow straightforwardly from the proof of Lemma 20 in \cite{Berard_Besson_Gallot:1994}. 

We prove (c) by induction. 
Start with $k=1$. Take $ Y \in F_1$ with $\| Y \|^2_{L^2(h)}=1$. Write $Y = X + X''$ with $X = \pi_1(Y)$. Then $\nabla^2_g X$ and $\nabla^2_g X''$ are orthogonal with respect to $L^2(g)$ since $\nabla^2_g X\in E_1$ and $\nabla^2_g X''\in E_1^\perp$. By (\ref{Calpha:Bound:Equation3}), it follows that 
\be
\label{approx}
(1-\epsilon') \left( \| \nabla^2_g X \|_{L^2(g)} + \| \nabla^2_g X''\|_{L^2(g)} \right) \leq \|\nabla^2_h Y\|_{L^2(h)}=\lambda_1(h) \leq \Lambda_1(1+\epsilon').
\ee
By (a), we know that $|\|X\|_g-\|Y\|_g|\leq \alpha_1(\epsilon)$ and $\|Y\|_g\leq (1+2\epsilon')\|Y\|_h=1+2\epsilon'$, which leads to the fact that $\|X\|_g=1+O(\epsilon)$.
The bounds (\ref{approx}) together with the fact $ \| \nabla^2_g X \|_{L^2(g)} = \Lambda_1 \|X\|_{L^2(g)} = \Lambda_1(1+O(\epsilon))$, when $\epsilon$ is small enough, we have 
\be
 \| \nabla^2_g X''\|_{L^2(g)} = O(\epsilon)
\ee
Assume that (c) is true for $F_1, F_2, \cdots, F_{k-1}$. Define $E'_k = \oplus_{i\leq k-1} E_i$ and $E''_k = \oplus_{i\geq k+1}E_i$. Take any $Y \in F_k$ such that $\| Y \|_{L^2(h)}=1$. Write $Y = X'+X+X''$ with respect to the decomposition $E'_k \oplus E_k \oplus E''_k$ of $L^2(g)$.  Note that $\nabla^2_g X'$, $\nabla^2_g X$ and $\nabla^2_g X''$ are orthogonal in $L^2(g)$, so we have
\be
(1 - \epsilon') \left( \Lambda_k \|X\|_{L^2(g)} + \| \nabla^2_g X' \|_{L^2(g)} + \|\nabla^2_g X'' \|_g \right) \leq \|\nabla^2_h Y \|_{L^2(h)} \leq \Lambda_k (1+\epsilon')
\ee
Since $ \Lambda_k \|X\|_{L^2(g)} = \Lambda_k(1 + O(\epsilon'))$ by the same argument as the above and $ \| \nabla^2_g X' \|_{L^2(g)} = O(\epsilon')$, it follows that 
$ \|\nabla^2_g X'' \|_g = O(\epsilon')$. 
\end{proof}

\begin{prop}
\label{prop:orthonoram_approx}
Let $h$ be a metric on $M$ such that $(1-\epsilon) \| g \|_{c^{2,\alpha}} < \| h\|_{c^{2,\alpha}} < (1+\epsilon)  \| g \|_{c^{2,\alpha}}$. There exist constants $\eta_{g,i}(\epsilon), 1\leq i \leq J$, which go to zero when $\epsilon$ approaches zero, such that to any orthonormal basis $\{ Y_i\}$ of eigenvector fields of $\nabla^2_h$ one can associate an orthonormal basis $\{ X_i \}$ of eigenvector fields $\nabla^2_g$ satisfying $\| X_i - Y_i \|_{H^1(g)} + \| \nabla^2_g X_i - \nabla^2_g Y_i \|_{L^2(g)} \leq \eta_{g,i}(\epsilon)$ for $i\leq J$.
\end{prop}

\begin{proof}
Take any orthonormal basis $\{ Y_i \}$ of eigenvector fields of $\nabla^2_h$. For each $k\leq k_0$ it defines an orthonormal basis of $F_k$, say $Y^k_1, \cdots, Y^k_{m_k}$. Applying the Gram-Schmidt process, we can orthonormalize the family $\pi_k(Y^k_{1}), \cdots, \pi_k(Y^k_{m_k})$ of $E_k$ to an orthonormal basis $X^k_{1}, \cdots, X^k_{m_k}$. By Lemma \ref{lem:error_approx} (a), $Y^k_1, \cdots, Y^k_{m_k}$ is $H^1(g)$-close to $\pi_k(Y^k_{1}), \cdots, \pi_k(Y^k_{m_k})$. 
By Lemma \ref{lem:error_approx} (b), the family $\pi_k(Y^k_{1}), \cdots, \pi_k(Y^k_{m_k})$ is almost orthonormal and hence $H^1(g)$-close to $X^k_1, \cdots, X^k_{m_k}$. 
Thus we know that $Y^k_1, \cdots, Y^k_{m_k}$ is $H^1(g)$-close to $X^k_1, \cdots, X^k_{m_k}$. By the same argument, $\nabla^2_g Y^k_1, \cdots, \nabla^2_gY^k_{m_k}$ and $\nabla^2_g X^k_1, \cdots, \nabla^2_g X^k_{m_k}$ are $L^2(g)$-close. This proves the Proposition. 
\end{proof}



To prove Theorem \ref{thm:ev_convergence}, we have to control the $L^p$ norm of the Hessian of a given vector field, and we need a Calderon-Zygmund type inequality, which is based the work of C. Wang \cite{Wang:2004}.
\begin{thm}\label{Theorem:Wang2004}
\cite[Theorem 0.1]{Wang:2004}
Let $(M^n,g), n \geq 3$, be a compact, connected, oriented Riemannian manifold without boundary. Suppose that 
\be
\mathrm{inj}(M) \geq i_0, \, \mathrm{Vol}(M) \leq V, \, \mbox{and} \, \left| \Ric \right| \leq \kappa.
\ee
For $1<q<\infty$, there exists a constant $C_q=C(q,n,K,i_0,V)$, such that for any $1$-form $\phi \in H^{\perp}$,
\be
\| \Hess \phi \|_{L^q(g)} \leq C_q \| \Delta_H \phi \|_{L^q(g)},
\ee 
where  $H$ is the space of harmonic $1$-forms and $\Delta_H$ is the Hodge Laplacian. 
\end{thm}

We are now ready to prove Theorem \ref{thm:ev_convergence}.
\begin{proof}
For $p > n$, by the Sobolev inequality \cite{Pal68} we have 
\be
\| Q \|_{C^{1,\alpha}(g)} \leq C^{(1)}_p(g) \left( \| \Hess_g Q \|_{L^p(g)} + \| \nabla_g Q\|_{L^p(g)} + \| Q \|_{L^p(g)} \right) 
\ee
for any $Q \in C^{\infty}(TM)$.
Note that vector fields are dual to $1$-forms and the Hodge Laplacian and the connection Laplacian only differ by the the Ricci curvature due to the Weitzeonb\"{o}ck formula. By Theorem \ref{Theorem:Wang2004}, for $W \in C^{\infty}(TM)\cap H^\perp$, we have
\be
\label{ineq:C-Z}
 \| \Hess_g W \|_{L^p(g)} \leq C \left( \| \nabla^2_g W \|_{L^p(g)} +  \| W \|_{L^p(g)} \right)
\ee
and hence
\be
\| W \|_{C^{1,\alpha}(g)} \leq C^{(2)}_p(g) \left( \| \nabla^2_g W \|_{L^p(g)} + \| \nabla_g W\|_{L^p(g)} + \| W \|_{L^p(g)} \right) \,.
\ee 

Let $\{Y_i \}$ be an orthonormal basis of eigenvector fields of $\nabla^2_h$ and let $\{ X_i \}$ be the orthonormal basis of eigenvector fields of $\nabla^2_g$ associated with $\{ Y_i \}$ given by Proposition \ref{prop:orthonoram_approx}. 
Denote $\| \cdot \|_{\infty} := \| \cdot \|_{L^{\infty}(g)}$ to be the sup-norm with respect to $g$. Denote $ \pi_0(X_i - Y_i) $ to be the harmonic part of $X_i-Y_i$ with respect to $g$, and $\pi^\perp_0(X_i - Y_i):=X_i - Y_i -\pi_0(X_i - Y_i)=X_i - (Y_i -\pi_0(Y_i))$. Clearly $ \pi_0(X_i - Y_i)=-\pi_0(Y_i)=-\langle X_0,Y_i\rangle_g X_0$. Thus, 
\begin{equation}\label{Proof:Perturbation:VectorField:Eq1}
\| X_i - Y_i \|_{C^{1,\alpha}(g)} \leq \| \pi_0(Y_i) \|_{C^{1,\alpha}(g)} + \|X_i - \pi^\perp_0(Y_i) \|_{C^{1,\alpha}(g)}. 
\end{equation}
By Proposition \ref{conj:est:sup} and Lemma \ref{lem:error_approx}(b), we immediately have that 
\begin{equation}\label{Proof:Perturbation:VectorField:Eq2}
\| \pi_0(Y_i) \|_{C^{1,\alpha}(g)} \leq C\|\pi_0(Y_i)\|_{L^2(g)}=\|\langle X_0,Y_i\rangle_gX_0\|_{L^2(g)}\leq \beta_k(\epsilon). 
\end{equation}
By the Sobolev embedding,   
\begin{align}
& \| X_i - \pi^\perp_0(Y_i) \|_{C^{1,\alpha}(g)} \label{Proof:Perturbation:VectorField:Eq3} \\
\leq& \,C^{(1)}_p(g) \left( \| \nabla^2_g (X_i - \pi^\perp_0(Y_i))  \|_{L^p(g)} + \| \nabla_g (X_i -\pi^\perp_0(Y_i)) \|_{L^p(g)} + \| X_i -\pi^\perp_0(Y_i) \|_{L^p(g)} \right) \nonumber\\
\leq&\, C^{(2)}_p(g) \left(  \| \nabla^2_g (X_i - \pi^\perp_0(Y_i))  \|_{L^2(g)} +\| X_i - \pi^\perp_0(Y_i) \|_{H^1(g)} \right)^{2/p} \nonumber\\
&\, \times \left( \| X_i \|_{\infty} + \| \pi^\perp_0(Y_i) \|_{\infty} + \|\nabla_g X_i \|_{\infty} +  \|\nabla_g \pi^\perp_0(Y_i) \|_{\infty} +\|\nabla^2_g X_i \|_{\infty} +  \|\nabla^2_g \pi^\perp_0(Y_i) \|_{\infty}\right) ^{(p-2)/p}\nonumber
\end{align}
where the last inequality follows by H\"{o}lder's inequality that 
\be
\| \nabla^2_g (X_i - \pi^\perp_0(Y_i))  \|_{L^p(g)}\leq \| \nabla^2_g (X_i - \pi^\perp_0(Y_i))  \|_{L^2(g)}^{2/p}\| \nabla^2_g (X_i - \pi^\perp_0(Y_i))  \|_\infty^{(p-2)/p}
\ee 
and
\begin{align}
&\| \nabla_g (X_i -\pi^\perp_0(Y_i)) \|_{L^p(g)} + \| X_i -\pi^\perp_0(Y_i) \|_{L^p(g)}\\
\leq&\, C\| X_i - \pi^\perp_0(Y_i) \|^{2/p}_{H^1(g)}(\| X_i - \pi^\perp_0(Y_i) \|_{\infty} + \|\nabla_g X_i - \nabla_g \pi^\perp_0(Y_i)\|_{\infty})^{(p-2)/p}\,.\nonumber
\end{align}
By the triangular inequality we have 
\begin{align}
&\| \nabla^2_g (X_i - \pi^\perp_0(Y_i))  \|_{L^2(g)} \label{Proof:Perturbation:VectorField:Eq4}\\
\leq&\, \| \nabla^2_g (X_i - Y_i)  \|_{L^2(g)} +\|\nabla^2_g \pi_0(Y_i) \|_{L^2(g)}
=\| \nabla^2_g (X_i - Y_i)  \|_{L^2(g)}\nonumber
\end{align} 
since $\nabla^2_g \pi_0(Y_i)=0$ and 
\begin{equation}
\| X_i - \pi^\perp_0(Y_i) \|_{H^1(g)}\leq \| X_i - Y_i \|_{H^1(g)}+\|\pi_0(Y_i) \|_{L^2(g)},\label{Proof:Perturbation:VectorField:Eq5}
\end{equation}
where we use the fact that $\|\pi_0(Y_i) \|_{H^1(g)}=\|\pi_0(Y_i) \|_{L^2(g)}$ since $\nabla_g\pi_0(Y_i)=0$. Further, note that 
\begin{equation}
\|\pi_0^\perp(Y_i)\|_\infty\leq \|Y_i\|_\infty+\|\pi_0(Y_i)\|_\infty\leq \|Y_i\|_\infty+\beta_k(\epsilon)\label{Proof:Perturbation:VectorField:Eq6}
\end{equation}
by (\ref{Proof:Perturbation:VectorField:Eq2}). Similarly, we have
\begin{equation}
\|\nabla_g\pi_0^\perp(Y_i)\|_\infty\leq \|\nabla_gY_i\|_\infty+\|\nabla_g\pi_0(Y_i)\|_\infty= \|\nabla_gY_i\|_\infty\label{Proof:Perturbation:VectorField:Eq7}
\end{equation}
and
\begin{equation}
\|\nabla^2_g\pi_0^\perp(Y_i)\|_\infty\leq \|\nabla^2_gY_i\|_\infty+\|\nabla^2_g\pi_0(Y_i)\|_\infty= \|\nabla^2_gY_i\|_\infty\label{Proof:Perturbation:VectorField:Eq8}
\end{equation}
since $\nabla_g\pi_0(Y_i)=\nabla^2_g\pi_0(Y_i)=0$.

By plugging (\ref{Proof:Perturbation:VectorField:Eq4}), (\ref{Proof:Perturbation:VectorField:Eq5}), (\ref{Proof:Perturbation:VectorField:Eq6}), (\ref{Proof:Perturbation:VectorField:Eq7}) and (\ref{Proof:Perturbation:VectorField:Eq8}) into (\ref{Proof:Perturbation:VectorField:Eq3}), together with (\ref{Proof:Perturbation:VectorField:Eq2}), (\ref{Proof:Perturbation:VectorField:Eq1}) becomes
\begin{align}
& \| X_i - Y_i \|_{C^{1,\alpha}(g)} \nonumber \\
\leq&\, C^{(2)}_p(g) \left(  \| \nabla^2_g (X_i - Y_i)  \|_{L^2(g)} +\| X_i - Y_i \|_{H^1(g)} + \beta_k(\epsilon)\right)^{2/p} \nonumber\\
&\, \times \left( \| X_i \|_{\infty} + \| Y_i \|_{\infty} + \|\nabla_g X_i \|_{\infty} +  \|\nabla_g Y_i \|_{\infty} +\|\nabla^2_g X_i \|_{\infty} +  \|\nabla^2_g Y_i \|_{\infty}+\beta_k(\epsilon)\right)^{(p-2)/p}\nonumber\,.
\end{align}
Note that by Proposition \ref{prop:orthonoram_approx}, for all $1 \leq i <k_0$  there exist constants $\eta_{g,j}(\epsilon)$ such that 
\be
\| X_i - Y_i \|_{H^1(g)} + \| \nabla^2_g (X_i - Y_i)  \|_{L^2(g)} \leq \eta_{g,j}(\epsilon).
\ee
Further, since the metric $h$ is $\epsilon$ close to $g$ in $c^{2,\alpha}$, the norms we consider for either the metric $g$ or the metric $h$ are comparable. Also note that $\|\nabla^2_g X_i \|_{L^{\infty}(g)} = \lambda_i(g) \|X_i \|_{L^{\infty}(g)}\leq C\|X_i \|_{L^{\infty}(g)}$ and $\|\nabla^2_h Y_i \|_{L^{\infty}(h)} = \lambda_i(h) \|Y_i \|_{L^{\infty}(h)}\leq C\|Y_i \|_{L^{\infty}(h)}$. Therefore, to complete the proof, it suffices to give uniform bounds for $ \| X_i \|_{L^{\infty}(g)}$,  $\| Y_i\|_{L^{\infty}(h)}$, $\|\nabla_g X_i \|_{L^{\infty}(g)}$, and $\|\nabla_h Y_i\|_{L^{\infty}(h)}$ for $i\leq k_0$.

By the assumption on the injectivity and volume, there exists an universal number $N_0$ so that $(M,g)$ and $(M,h)$ are covered by $N_0$-balls. Indeed, due to the lower bound of the injectivity, by Croke's lemma \cite{Croke:1980}, the balls of radius $\epsilon$ is bounded from below; on the other hand, since the volume is bounded by $V$, we have a bounded number of balls. Thus, by applying Proposition \ref{conj:est:sup}, there exists some constant $C=C(i_0, V, \alpha, n, Q,\kappa, k_0)$ so that 
\be
\| X_i \|_{L^{\infty}(g)}, \| Y_i \|_{L^{\infty}(h)}, \|\nabla_g X_i \|_{L^{\infty}(g)}, \mbox{ and } \|\nabla_h Y_i \|_{L^{\infty}(h)} \leq C
\ee
for $i\leq k_0$.
\end{proof}

\section{Examples}
\label{examples}
In this section, we illustrate how VDM is carried out numerically. A detailed discussion of the algorithm and its application to different problems could be found in \cite{Singer_Wu:2012,Zhao_Singer:2014,Chung_Kempton:2013,Alexeev_Bandeira_Fickus_Mixon:2013,Marchesini_Tu_Wu:2014}, and its robustness analysis could be found in \cite{ElKaroui_Wu:2015b}

Suppose we sampled $n$ points $\mathcal{X}:=\{x_i\}_{i=1}^m$ uniformly from a $n$-dim smooth manifold $M$ without boundary embedded in $\mathbb{R}^{p}$ via $\iota$; that is, $\{x_i\}_{i=1}^m\subset \iota(M)\subset \RR^p$. Assume that $M$ is endowed with a Riemannian metric $g$ induced from the ambient Euclidean space via $\iota$. We construct the graph connection Laplacian (GCL) $C\in \RR^{mn\times mn}$ in the following way. 

\begin{itemize}
\item Step 1: build up an undirected graph $G=(V,E)$, where $V= \mathcal{X}$ and $E$ is the edges $E=\{(x_i,x_j)|\, i\neq j\}$. Note that there are different ways to design $E$ from a given point cloud. To demonstrate the idea, here we simply design the edge so that the undirected graph is complete.

\item Step 2: establish the {\it affinity function}, $w:E\to \RR^+$, by setting $w(x_i,x_j)=K(\|x_i-x_j\|_{\RR^p}/\sqrt{\epsilon})$, where $\epsilon>0$ is the ``bandwidth'' and $K$ is a suitable kernel function chosen by the user. In all experiments in this section, we used the kernel function $K(u)=e^{-5u^2}\chi_{[0,1]}$ for the definition of the weight function $w$. 

\item Step 3: for $x_i$, choose a basis $\{u_{i,1},\ldots,u_{i,n}\}\subset \RR^p$, of the embedded tangent space, $\iota_*T_{x_i}M$, which is a $d$-dim affine subspace in $\RR^p$. Denote $B_i=[u_{i,1},\ldots, u_{i,n}]\in \RR^{p\times n}$. In general, we do not have the information of $\iota_*T_{x_i}M$, and need to estimate $B_i$ via the local principal component analysis based on the fact that locally the manifold could be well approximated by the embedded tangent space \cite{Singer_Wu:2012}.

\item Step 3: establish the {\it connection function}, $g:E\to O(d)$, where $g(x_i,x_j)=O_{ij}$ and $O_{ij}=\arg\min_{O\in SO(n)}\|B_i^TB_j-O\|$. Here $O_{ij}$ is an approximation of the parallel transport from $x_j$ to $x_i$ associated with the induced metric $g$ \cite{Singer_Wu:2012}. Note that in general, due to the curvature, $B_i^TB_j$ is not an orthogonal matrix, so we need an optimization step to obtain the parallel transport.

\item Step 4: establish a $n\times n$ block matrix $S$, with $d\times d$ block entries, so that the $(i,j)$-th block is $S[i,j]=w(x_i,x_j) g(x_i,x_j)$; establish a $n\times n$ diagonal block matrix $D$, with $d\times d$ block entries, so that the $(i,j)$-th block is $D[i,j]=\sum_{l=1}^m w(x_i,x_l) I_{n\times n}\delta_{ij}$

\item Step 5: establish the GCL $C=I_{mn\times mn}-D^{-1}S$. 
\end{itemize}

With the GCL $C$, we could evaluate its eigenvalue $\lambda_l$ and eigenvectors $u_l\in\RR^{mn}$, and order the eigenvalues in the increasing order: $\lambda_1\leq \lambda_2\leq\ldots$. Denote $u_i[l]\in\RR^n$ as $u_i[l]=(u_i((l-1)d+1),\ldots,u_i(ln))^T$. The VDM with the diffusion time $t>0$ for the sampled point cloud $\{x_i\}_{i=1}^n$ is defined as
\begin{equation}
V_t:\, x_i\mapsto (e^{-(\lambda_k+\lambda_l)t}\langle u_k[i],u_l[j]\rangle)_{l,k=1}^{mn}\subset \RR^{(mn)^2}.
\end{equation}
Here, we emphasize that $u_i[l]$ is actually the {\em coordinate} of the $i$-th eigenvector field at $x_l$ {\em associated with the basis chosen for $T_{x_l}M$} \cite{Singer_Wu:2016}. It has been shown that asymptotically, the GCL will converge not only pointwisely but also spectrally to the connection Laplacian of $M$. Thus, we could numerically implement VDM. Note that in general we could not visualize a geometric object with dimension higher than $3$, so in the following numerical implementation, we illustrate the truncated VDM by choosing suitable subset of embedding coordinates.

The first example is taking the manifold $M$ to be diffeomorphic to $S^1$ embedded in $\RR^p$. Here, $M$ is generated by the X-ray transform of a $2$-dim image compactly supported on $\RR^2$. The dataset is generated in the following way. Take the image of interest to be a compactly supported function $f$ so that $\text{supp}(f)\subset B_1$, where $B_1$ is the ball with center $0$ and radius $1$. Sample uniformly $m$ points on $S^1$, denoted as $\{\theta_i\}_{i=1}^m$. The dataset is generated by discretizing the X-ray transform by $x_i:=[R_\psi(\theta_i,z_j)]_{i=1}^p$, where $i=1,\ldots, n$, $R_\psi(\theta_i,z_j)=\int \psi(y\theta_i+z_j)d y$, $z_j=-1+2j/p$ and $p$ is chosen big enough so that the topology of $R_\psi(S^1)$ is preserved after discretization. Note that while we sample uniformly from $S^1$, after the X-ray transform $R_\psi$, which is a diffeomorphism from $S^1$ to a $1$-dim manifold $M$, the dataset $\{x_i\}_{i=1}^m$ is non-uniformly sampled from $M$. We refer the reader with interest to \cite{Singer_Wu:2013a} for more details about this 2-dim random tomography problem. 
Take $m=2000$. Note that since the tangent bundle of $M$, which is diffeomorphic to $S^1,$ is trivial, the connection Laplacian is reduced to the Laplace-Beltrami operator. Thus, the first eigenvector field is the constant function, the second and third eigenvector field are the sine and cosine functions, and so on. We thus consider the following truncated VDM
\begin{equation}
V^{(2)}_1: x_i\mapsto (e^{-(\lambda_1+\lambda_2)}\langle u_1[i],u_2[i]\rangle, e^{-(\lambda_1+\lambda_3)}\langle u_1[i],u_3[i]\rangle)^T\in \RR^2,
\end{equation}
which embeds $M$ into $\RR^2$. Figure \ref{fig1} shows this truncated VDM. Note VDM generates a non-canonical circle embedded in $\RR^2$, due to the non-uniform distribution on $M$.

\begin{figure}[h]
\includegraphics[width=0.7\textwidth]{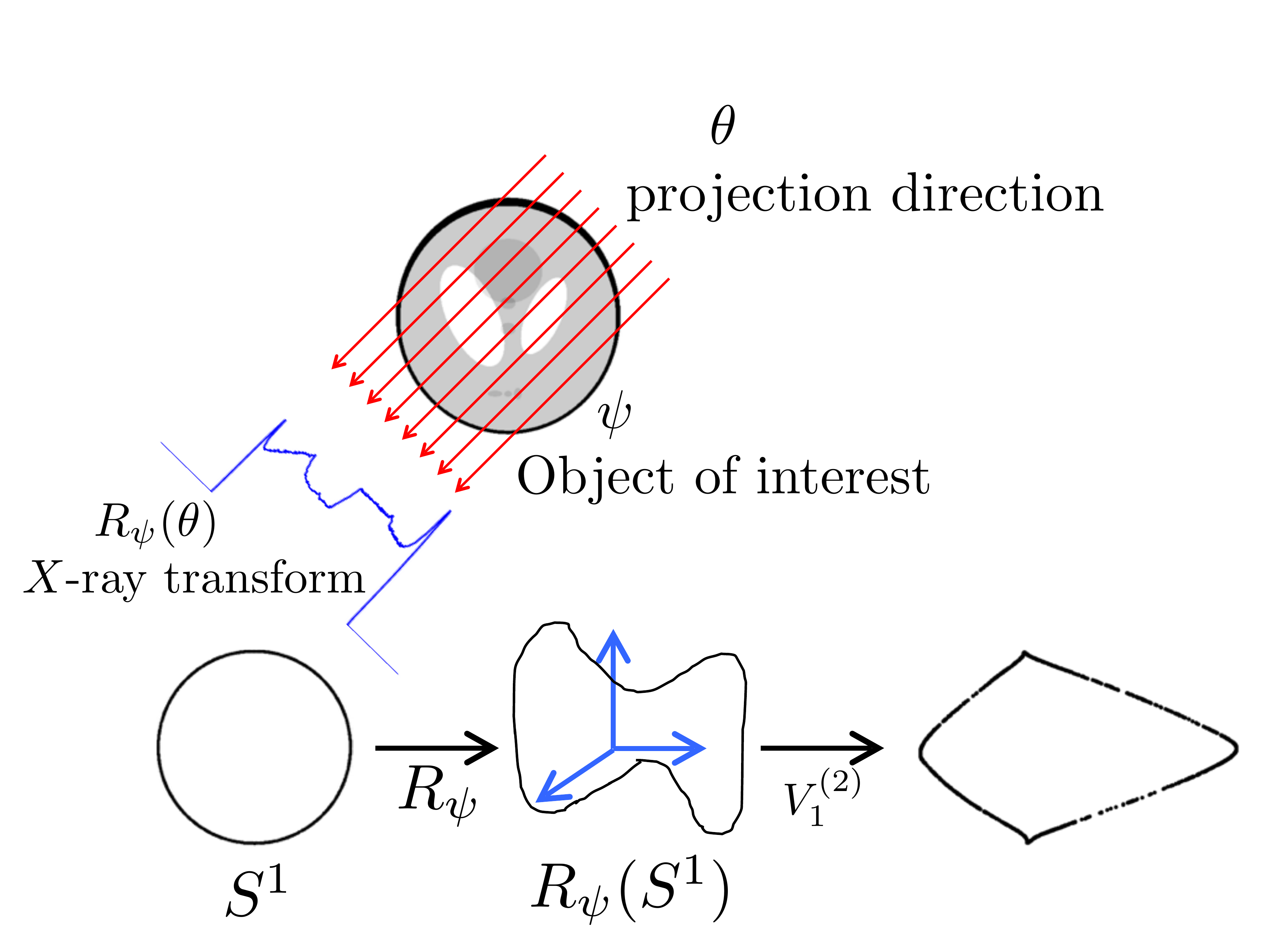}
\caption{A nontrivial $S^1$ generated by the X-ray transform and the VDM result.}
\label{fig1}
\end{figure}

The second example shows that we might not be able to visualize $S^2$ by VDM. Take the manifold to be the canonical $S^2$ embedded in $\RR^3$. Take $m=2000$. Note that although the tangent bundle of $S^2$ is nontrivial, its well-known that the first few eigenvector fields could be easily evaluated -- the eigenvector fields associated with the smallest eigenvalue, $\lambda_1$, of the connection Laplacian, denoted as $X_1$, $X_2$ and $X_3$, are nothing but the project of the constant vector fields $(1,0,0)$, $(0,1,0)$ and $(0,0,1)$ on $\RR^3$ onto the embedded tangent space. Also, since $S^2$ is a Riemannian surface with the complex structure, the ``complex conjugate'' of $X_1$ (respectively $X_2$ and $X_3$), denoted as $X_4$ (respectively $X_5$ and $X_6$), is also an eigenvector field associated with $\lambda_1$. We carry out the following calculation to explicitly evaluate these eigenvector fields. Denote
\begin{equation}
R=\begin{bmatrix} | & | & |\\R^1 & R^2 & R^3\\|& | & | \end{bmatrix}=\begin{bmatrix} R^1_1 & R^2_1 & R^3_1\\R^1_2 & R^2_2 & R^3_2\\R^1_3 & R^2_3 & R^3_3\end{bmatrix}\in SO(3).
\end{equation}
Since $SO(3)$ is the frame bundle of $S^2$, if we view $R^3\in S^2$ as a sampled point from $S^2$, $R^1$ and $R^2$ form a basis of the embedded tangent plane of $T_{R^3}S^2$, which is an affine subspace of $\RR^3$. By a direct calculation, we know that $\iota_*X_1$, the embedded vector field $X_1$ in $\RR^3$, becomes
\begin{equation}
\iota_*X_1(R^3)=\begin{bmatrix} | & | \\R^1 & R^2 \\|& | \end{bmatrix}\begin{bmatrix} | & | \\R^1 & R^2 \\|& | \end{bmatrix}^T\begin{bmatrix} 1 \\0\\0 \end{bmatrix}=\begin{bmatrix} | & | \\R^1 & R^2 \\|& | \end{bmatrix}\begin{bmatrix} R^1_1 \\R^2_1\end{bmatrix}.
\end{equation}
Similarly, for $k=2,3$, we know that 
\begin{equation}
\iota_*X_k(R^3)=\begin{bmatrix} | & | \\R^1 & R^2 \\|& | \end{bmatrix}\begin{bmatrix} R^1_k \\R^2_k\end{bmatrix}.
\end{equation}
By the cross product associated with the almost complex structure on $S^2$, $X_4,X_5$ and $X_6$ could be evaluated directly. For $k=1,2,3$, we have
\begin{equation}
\iota_*X_{k+3}(R^3)=R^3\times \iota_*X_k(R^3)=\begin{bmatrix} | & | \\R^2 & -R^1 \\|& | \end{bmatrix}\begin{bmatrix} R^1_k \\R^2_k\end{bmatrix}=R^1_kR^2-R^2_kR^1.
\end{equation}
Thus, 
\begin{align}
\langle X_k(R^3),X_l(R^3)\rangle &= \langle X_{k+3}(R^3),X_{l+3}(R^3)\rangle=R^1_kR^1_l+R^2_kR^2_l, \nonumber\\
\langle X_k(R^3),X_{l+3}(R^3)\rangle &= -R^1_kR^2_l+R^2_kR^1_l,
\end{align}
for $k,l=1,\ldots,3$.
By a direct calculation, we have $\langle X_k(R^3),X_l(R^3)\rangle=\langle X_k(-R^3),X_l(-R^3)\rangle$ and $\langle X_k(R^3),X_{l+3}(R^3)\rangle=-\langle X_k(-R^3),X_{l+3}(-R^3)\rangle$. Further, when $k\neq l$, the map $R^3\mapsto \langle X_k(R^3),X_{l+3}(R^3)\rangle$ maps $\pm e_k$ and $\pm e_l$ to $0$. Thus, the truncated VDM depending only on a randomly chosen three $\langle X_k(R^3),X_l(R^3)\rangle$, $k,l=1,\ldots,6$, is not an embedding. Note that this is different from the diffusion maps -- we could simply visualize $S^2$ via the truncated diffusion maps by taking the top three eigenfunctions into account. Also, we know that the range of the map
\begin{equation}
V^{(3)}_1: x_i\mapsto (e^{-2\lambda_k}\langle u_k[i],u_k[i]\rangle)_{k=1}^3\in \RR^3
\end{equation}
is diffeomorphic to $\RR P^2$.

The theoretical analysis and the numerical results open the following questions. Although the theorem shows that each compact smooth manifold without boundary could be embedded to a finite dimensional Euclidean space by the truncated VDM taking finite eigenvector fields into account, the number of necessary eigenvector fields is not clear. The same problem happens when we discuss the spectral embedding by the truncated diffusion map. 
Even worse, while the purpose of many nonlinear embedding algorithms is to ``reduce the dimension'', there is no guarantee that we could achieve it if we could not know the necessary number of the eigenvector fields. Recall that by ``reduce the dimension'', we mean that the ``dimension'' of the sampled point cloud could be well represented by another point cloud embedded in a lower dimensional space. Without the knowledge of the necessary number of eigenvector fields, we could not rule out the possibility that the number of eigenvector fields we need for the truncated VDM (or the number of eigenfunctions in the truncated DM) might be more than the ambient space dimension of the original dataset. We thus need to find a way to estimate the necessary number of the eigenvector fields (respectively eigenfunctions) for the truncated VDM (respectively DM). Ideally we would like to have a bound of this number based on the basic geometric/topological quantities. This problem is prevalent in the field actually -- for a given nonlinear embedding algorithm, how to evaluate if the dimension reduction mission for a dataset could be achieved? To the best of our knowledge, this is also an open problem for most nonlinear embedding algorithm.


\bibliography{noisyManifold}
\bibliographystyle{plain}

\end{document}